\documentclass{cimart}

\let\opn\operatorname
\newcommand\blank{\mathord{\hbox to 1ex{\hrulefill}}}

\vspace{1.5em}

\title[Transposed Poisson structure on the Witt-type algebra $\mathcal{W}(a,-1)$]{Transposed Poisson structure on the Witt-type algebra $\mathcal{W}(a,-1)$: derivations, automorphisms, and Rota--Baxter operators}

\authors{Roman Lubkov and Ilya Shiringovskiy}

\authorinfo[R. Lubkov]{Saint Petersburg State University, Russia}{RomanLubkov@yandex.ru, r.lubkov@spbu.ru}

\authorinfo[I. Shiringovskiy]{Saint Petersburg State University,  Russia}{shiringovskiy@gmail.com}

\abstract{%
    In this paper, we provide a comprehensive study of the structural properties of the transposed Poisson algebra $\mathcal{W}(a,-1)$. We classify several types of linear maps, including derivations, local derivations, quasi-derivations, and $\delta$-derivations, showing that non-trivial $\delta$-derivations exist only for $\delta=1$ and $\delta=\frac{1}{2}$. Furthermore, we describe the groups of automorphisms, local automorphisms, 2-local automorphisms, and quasi-automorphisms.

    We also investigate Rota--Baxter operators of weight $1$ on $\mathcal{W}(a,-1)$. Specifically, we classify operators that are homogeneous with respect to both the standard $\mathbf{Z}$-grading and a $\mathbf{Z}_2$-grading, establishing a rigidity result for the latter case. Finally, we classify all $\mathbf{W}$-compatible Novikov--Poisson structures, demonstrating that the associative product on the Witt algebra is universally compatible with its known Novikov structures.
    }



\VOLUME{34}
\ISSUE{1}
\NUMBER{14}
\DOI{https://doi.org/10.46298/cm.18245}
\editinfo{May 20, 2026}{June 7, 2026}{Ivan Kaygorodov}

\acknowledgments{%
    The authors would like to thank the SRMC (Sino-Russian Mathematics Center in Peking University, Beijing, China) for its hospitality and excellent working conditions, where some parts of this work have been done. The work is supported by RSF 25-21-00545.
}

\begin{document}

\section{Introduction}
\subsection{Transposed Poisson algebras}
Poisson algebras emerged from the Poisson brackets of Hamiltonian mechanics. These structures are widely studied in symplectic geometry, the theory of integrable systems, and classical and quantum mechanics~\cite{Sinha_Yadav_PoissonGeometric_2024, Yoshida_NambuMechanics_2024, Mikhailov_Vanhaecke_CommutativePoisson_2024, Landsman_ClassicalQuantum_1998}. Extensions of this concept, particularly in the context of hydrodynamic-type brackets and Hamiltonian operators, led to the discovery of Novikov–Poisson algebras~\cite{IntroXu}. These algebras are connected to Poisson brackets of hydrodynamic type and the classification of Hamiltonian operators~\cite{IntroGelf}, and they describe specific aspects of fluid dynamics and integrable systems.

The structure of a transposed Poisson algebra emerged from the study of Novikov--Poisson algebras. In particular, the commutator of a Novikov product, $[x,y]=x\circ y - y\circ x$, defines a Lie bracket. While it is natural to expect this bracket, together with a commutative associative product, to form a standard Poisson algebra, these operations satisfy the transposed Leibniz identity~\cite{BBGW23}. Furthermore, the Novikov product can be modified by an associative term: for any fixed element $v$ in the algebra, the new operation defined by$$x\circ'y := x\circ y - v\cdot x \cdot y$$provides a new Novikov--Poisson structure $(\mathfrak{L}, \circ', \cdot)$ on the original algebra~\cite{IntroXu}.

This algebraic structure appears naturally in several other settings. For instance, given a derivation $D$ on a commutative associative algebra, a transposed Poisson bracket can be defined by:
$$[x,y]:=x\cdot D(y) - y \cdot D(x).$$
\noindent
Furthermore, S.~Chen and C.~Bai.~\cite{IntroChen} showed that transposed Poisson algebras arise as classical limits of Novikov deformations. In this framework, the Lie bracket is obtained as:
$$[x,y]:=\frac{x \cdot_\hslash y - y\cdot_\hslash x}{\hslash} \pmod{\hslash}$$
where $x\cdot_\hslash y \equiv x \cdot y \pmod{\hslash}$ over the base field of characteristic $0$.

These algebras also have natural generalizations and connections to other algebraic theories. One active research direction, introduced in~\cite{AKS26}, involves adding a parameter to the definitions of Poisson and transposed Poisson algebras. This results in $\delta$-Poisson and transposed $\delta$-Poisson algebras, where the identities linking the operations are scaled by $\delta$. Recent studies, such as the classification of transposed $\delta$-Poisson structures on null-filiform algebras~\cite{DQT26} and the work on the $\delta$-first Whitehead Lemma for Jordan algebras~\cite{ZZ24}, demonstrate how deformation and cohomological techniques apply in this broader context. Additionally, B.~Sartayev~\cite{Sartayev2023} has investigated some generalizations of the variety of transposed Poisson algebras, establishing their basic identities and comparing them with classical varieties. The landscape of non-associative structures, including results relevant to Poisson-type algebras, has been comprehensively surveyed by I.~Kaygorodov~\cite{k23}. Finally, connections to other areas, such as skew braces and Rota--Baxter operators on semi-direct products~\cite{BS26}, suggest links between transposed Poisson algebras and the theory of solutions to the Yang--Baxter equation, opening new directions for future research.

\medskip
In the present paper, we study the algebra $\mathcal{W}(a,b)$, which is the semi-direct product of the Witt algebra $\mathbf{W}$ and the tensor density module over it --- $I(a,b)$. We focus on the case $b=-1$. This choice is motivated by the work of Ferreira, Kaygorodov, and Lopatkin~\cite{FKL_21}, who proved that all $\frac{1}{2}$-derivations of this Lie algebra are trivial when $b \neq -1$. According to~\cite{IntroIvan}, this implies that non-trivial transposed Poisson structures cannot exist in that case.

In fact, finding non-trivial $\frac{1}{2}$-derivations has become a common and effective way to identify such structures. For instance, recent studies have computed all $\frac{1}{2}$-derivations for deformative Schrödinger--Witt algebras, as well as non-finitely graded Witt and Heisenberg--Witt algebras, thereby classifying all their transposed Poisson structures~\cite{KKS2025} (see more cases in~\cite{AAES_24, YXK_23, KKS25}).

Beyond these specific classifications, it is worth noting that generalized structures are just as significant as classical ones. For example, generalized derivations, introduced by G.~Leger and E.~Luks in~\cite{Leger2000}, were recently used by R.~Andruszkiewicz, T.~Brzezi{\'{n}}ski, and K.~Radziszewski in~\cite{ABR25}. They established a one-to-one correspondence between Lie affgebras and pairs $(\mathfrak{g}, D)$, where $\mathfrak{g}$ is a Lie algebra and $D$ is a generalized derivation (supplemented by an element of $\mathfrak{g}$). This result underscores the deep connections between generalized derivations, non-associative structures, and the development of transposed Poisson algebras.

Parallel to the study of generalized operators, the concept of ``local'' structures has also gained prominence. This idea originally migrated into algebra from functional analysis. In some finite-dimensional algebras, local structures do not coincide with global ones. A notable example is the Cayley algebra $\mathcal{C}$ over a field $k$ with norm $n$. In \cite{AEK2023}, S.~Ayupov, A.~Elduque, and K.~Kudaybergenov established the following result:
\begin{itemize}
    \item The local derivations of $\mathcal{C}$ form the Lie algebra $\{ d \in \mathfrak{so}_n(\mathbf{C}) \mid d(1) = 0 \}$;
    \item The $2$-local derivations coincide with the local derivations.
\end{itemize}

\subsection{Structure of the paper}
The paper is organized as follows. \textbf{Section 2} focuses on linear maps related to the Lie structure. We classify derivations, local derivations, and quasi-derivations. A key result here is that non-trivial $\delta$-derivations only exist for $\delta=1$ and $\delta=\frac{1}{2}$. \textbf{Section 3} describes the symmetries of the algebra. We classify automorphisms and their generalizations, such as local and 2-local automorphisms. \textbf{Section 4} deals with Rota--Baxter operators. We classify operators that are homogeneous with respect to the $\mathbf{Z}$-grading and prove a rigidity result for the $\mathbf{Z}_2$-grading. Finally, \textbf{Section 5} examines $\mathbf{W}$-compatible Novikov--Poisson structures. We show that the associative product on the Witt algebra is compatible with all known Novikov structures on $\mathbf{W}$.

\subsection{Main results}
In this work, we classify both classical structures, such as derivations and automorphisms, and their various generalizations, including quasi-automorphisms and quasi-derivations. Our results show the structural rigidity of $\mathcal{W}(a,-1)$ by establishing that ``local'' versions of these maps actually coincide with the global ones. Specifically, we show that all local and 2-local structures are identical to the standard ones. Furthermore, we find that while the $\mathbf{Z}$-grading allows for several types of Rota--Baxter operators, the $\mathbf{Z}_2$-grading is so restrictive that it allows only the trivial operator.

\medskip
The semi-direct product $\mathcal{W}(a,-1) = \mathbf{W} \ltimes I(a,-1)$ yields highly coupled algebraic systems due to the coexistence of the Lie bracket and the associative product. To solve these asymmetric matrix equations, we use a functional reduction method. It expresses one multiplication directly through the other, turning hard operator conditions into simpler scalar equations.

Our approach helps us manage structural problems in the module. These problems create an issue between periodicity (cosets) and sparsity when we classify Rota--Baxter operators in Theorem~\ref{th:RB-0-w-Z-grad-Lie}. We solve this issue by using sequence selectors to study and restrict the operator's support.

\subsection{Notation}
In this paper, all algebras and linear spans are considered over the field of complex numbers $\mathbf{C}$. For the tensor density module over $\mathbf{W}$, $I(a,-1)$, we set $I_a = 0$ if the parameter $a$ is not an integer. Similarly, for any map $f$ with an integer domain, we define $f(a) = 0$ whenever $a \notin \mathbf{Z}$. 

Let us define the primary objects of this study. Let $\mathfrak{L}$ be a linear space equipped with two bilinear operations: 
$$\cdot,\; [\ ,\ ]\colon \mathfrak{L} \times \mathfrak{L} \longrightarrow \mathfrak{L}.$$

\begin{definition}\label{def:tPa}
    The triple $\left(\mathfrak{L},\cdot,[\ ,\ ]\right)$ is called a \textbf{transposed Poisson algebra} if $(\mathfrak{L},\;\cdot)$ is a commutative and associative algebra, $(\mathfrak{L},\;[\ ,\ ])$ is a Lie algebra, and the transposed Leibniz identity holds: 
    $$2z\cdot [x,y]=[z\cdot x, y] +[x,z\cdot y].$$
\end{definition}

Let $\mathbf{W} = \bigl\langle L_n\mid n\in\mathbf{Z}\bigr\rangle$ be the Witt algebra and $I(a,-1) = \bigl\langle I_m\mid m\in\mathbf{Z}\bigr\rangle$ be the tensor density module over $\mathbf{W}$. On the semi-direct product $\mathcal{W}(a,-1) = \mathbf{W}\ltimes I(a,-1)$ we define the following relations on the generators $L_n$ and $I_m$:
\begin{align*}
    L_n\cdot L_m &:= L_{n+m} & I_n\cdot I_m &:= 0 &  L_n \cdot I_m &:= I_{n+m},\\
    [L_n,L_m] &:= (n-m)L_{n+m} & [I_n, I_m] &:= 0 & [L_n,I_m] &:= -(m+a-n)I_{n+m},
\end{align*}
for all integers $n$ and $m$. These relations define a \textbf{transposed Poisson algebra structure on $\mathcal{W}(a,-1)$}.

\section{Derivations and Related Operators}
\subsection{Derivations}
X.~Tang described the derivations of the Lie structure $(\mathcal{W}(a,-1), [\ ,\ ])$ in \cite{XTang}. Every such derivation is a sum of an inner derivation and a projector onto the tensor density module over $\mathbf{W}$:
$$ \mathfrak{Der}_{\opn{Lie}}\bigr(\mathcal{W}(a,-1)\bigl)=\mathfrak{Inn}\bigr(\mathcal{W}(a,-1)\bigl)\oplus\,\mathbf{C}\,D_1, $$
where $D_1(L_m) = 0,\; D_1(I_m) = I_m$.
Derivations of the transposed Poisson structure must be compatible with both products. Let $D$ be an arbitrary derivation of the Lie structure: $$D = \mathrm{ad}_\xi + \lambda D_1,$$ where $\xi = \sum\limits_j \alpha_jL_j + \beta_jI_j$. We obtain the following classification theorem:

\begin{theorem}\label{th:derW_a_b}
    The operator $D$ is a derivation of the transposed Poisson algebra $\mathcal{W}(a,-1)$ if and only if $\xi = \alpha L_0 + \beta I_{-a}$ and $\lambda = 0$ when $a\in\mathbf{Z}$. In other words for any complex $a$,
    $$
        \mathfrak{Der}\bigl(\mathcal{W}(a,-1)\bigr) = \mathbf{C}\;\mathrm{ad}_{L_0} + \mathbf{C}\;\mathrm{ad}_{I_{-a}} + \mathbf{C}\; D_1.
    $$ 
\end{theorem}
\begin{remark}
     In operator terms, these derivations are:
    \begin{align*}
        \begin{pmatrix}
            \mathrm{diag}\bigl(-n\mid n\in\mathbf{Z}\bigr) & 0 \\
            0 & \mathrm{diag}\bigl(-(m+a)\mid m\in\mathbf{Z}\bigr)
        \end{pmatrix},~\begin{pmatrix}
            0 & 0 \\
            \bigl(i\delta_{ i,i+a}(i,j)\bigr)_{i,j} & 0
        \end{pmatrix},~\begin{pmatrix}
                0 & 0 \\
                0 & \mathrm{Id}
            \end{pmatrix}.
    \end{align*}   
\end{remark}

\begin{proof}
The derivation satisfies the following condition:
$$[\xi,x]\cdot y + \lambda D_1(x)\cdot y + x\cdot [\xi,y] + x\cdot \lambda D_1(y) = [\xi,x\cdot y] + \lambda D_1(x\cdot y).$$
Since $D$ is a linear operator, we prove this for the generators of our algebra.\\
\underline{Case 1}: $x = L_n$ and $y = L_m$.
\begin{multline*}
    \sum_j \bigl((j-n)\alpha_j + (j-m)\alpha_j\bigr)L_{n+m+j} + \sum_j \bigl((j+a-n)\beta_j + (j+a-m)\beta_j\bigr)I_{n+m+j} \\
    = \sum_j (j-n-m) \alpha_j L_{n+m+j} + (j+a-n-m)\beta_j I_{n+m+j}.
\end{multline*}
Comparing the coefficients, we obtain:
\begin{align*}
    j\alpha_j &= 0,\\
    (j+a)\beta_j &= 0.
\end{align*}
This implies:
\begin{align}
    \alpha_j &= \alpha\delta_{0,j}, \quad \alpha\in\mathbf{C},\\
    \beta_j &= \beta\delta_{-a,j}, \quad \beta\in\mathbf{C}.
\end{align}
\noindent
\underline{Case 2}: $x = L_n$ and $y = I_m$.
\begin{align*}
    n\alpha I_{n+m} + (m+a)\alpha I_{n+m} + \lambda I_{n+m} = \alpha(n+m+a)I_{n+m}.
\end{align*}
After computation, we get:
\begin{align*}
    \alpha\bigl(m+a+n - (m+n+a)\bigr)I_{n+m} + \lambda I_{n+m} = \lambda I_{n+m}.
\end{align*}
Finally, we obtain:
$$
D(\blank) = [\alpha L_0 + \beta I_{-a}, \blank] + \lambda D_1(\blank).
$$
\end{proof}

Our results show that the associative product strongly restricts the derivations of the algebra. For the pure Lie structure of the original deformative Schr\"{o}dinger--Virasoro algebra $\mathbf{W}^g(a,b)$, Q.~Jiang and S.~Wang~\cite{JQWS} found a large derivation algebra. Their classification depends heavily on the specific choices of the parameters $a$ and $b$. In contrast, the associative product in our transposed Poisson algebra $\mathcal{W}(a,-1)$ eliminates most of those derivations. This compatibility forces a much more rigid result. Furthermore, many inner Lie derivations do not preserve the associative multiplication. As a result, the standard Lie algebra formula $H^1 = \mathfrak{Der} / \mathfrak{Inn}$ does not apply directly to transposed Poisson structures.

\begin{corollary}
Consider the transposed Poisson algebra structure on the Witt algebra induced from $\mathcal{W}(a,-1)$. The derivations are described as:
$$
\mathfrak{Der}(\mathbf{W}) = \mathbf{C}\,\mathrm{ad}_{L_0},
$$
or in operator form:
$$
\mathfrak{Der}(\mathbf{W}) = \bigl\langle \mathrm{diag}\bigl(-n \mid n \in \mathbf{Z} \bigr) \bigr\rangle.
$$
\end{corollary}

\subsection{$\frac{1}{2}$-derivations}

The class of $\frac{1}{2}$-derivations is important for transposed Poisson algebras. The transposed Leibniz rule implies that the map $y \mapsto x \cdot y$ is a $\frac{1}{2}$-derivation. This is analogous to how the map $y \mapsto [x, y]$ is a derivation for a Poisson algebra.
\begin{definition}\label{def:1/2-der}
    A linear map $D$ is called a \textbf{$\frac{1}{2}$-derivation} if it satisfies the following condition:
    \begin{align*}
        D(x\star y) &= \frac{1}{2}\bigl(D(x) \star y + x\star D(y)\bigr),
    \end{align*}
    for all elements $x, y$ and all products $\star$ in the algebra.
\end{definition}

In~\cite{FKL_21}, the authors established that any element $D \in \frac{1}{2}\textrm{-}\mathfrak{Der}_{\opn{Lie}} \bigl(\mathcal{W}(a,-1)\bigr)$ acts on the generators as follows:
$$D(L_n) = \sum_j \alpha_j L_{j+n} + \beta_j I_{j+n}, \quad
   D(I_m) =\sum_j\alpha_j I_{j+m},$$
where the sequences $\alpha_j, \beta_j \subset \mathbf{C}$ have finite support.
\begin{theorem}\label{th:1/2-derW_a_b}
    Every $\frac{1}{2}$-derivation $D \in \frac{1}{2}\textrm{-}\mathfrak{Der} \bigl(\mathcal{W}(a,-1)\bigr)$ acts on the generators according to:
    $$D(L_n) = \sum_j \alpha_j L_{j+n} + \beta_j I_{j+n}, \quad
    D(I_m) =\sum_j\alpha_j I_{j+m},$$
    for sequences $\alpha_j, \beta_j \subset \mathbf{C}$ with finite support.
\end{theorem}
\begin{remark}
    In operator terms, the $\frac{1}{2}$-derivations of $\mathcal{W}(a,-1)$ are generated by:
    \begin{align*}
        A_{t \in \mathbf{Z}} := \begin{pmatrix}
            \mathcal{J}^t & 0 \\
            0 & \mathcal{J}^t
        \end{pmatrix}, \quad B_{t \in \mathbf{Z}} := \begin{pmatrix}
            0 & 0 \\
            \mathcal{J}^t & 0
        \end{pmatrix}.
    \end{align*}    
\end{remark}

\begin{proof}
    We verify the condition for the generators $L_n$ and $I_m$.\\
    \underline{Case 1}: $x = L_n$ and $y = L_m$.
    \begin{multline*}
        \sum_j \alpha_j L_{j+n+m} + \beta_j I_{j+n+m} = \frac{1}{2} \bigg( \sum_j \alpha_j L_{j+n+m} + \beta_j I_{j+n+m}
        + \sum_j \alpha_j L_{j+m+n} + \beta_j I_{j+m+n} \bigg).
    \end{multline*}
    The condition holds automatically, so this case imposes no additional restrictions.\\
    \noindent
    \underline{Case 2}: $x = L_n$ and $y = I_m$. Here we have
    \begin{align*}
        D(L_n \cdot I_m) &= \sum_j \alpha_j I_{j+n+m}, \\
        D(L_n) \cdot I_m + L_n \cdot D(I_m) &= \sum_j 2\alpha_j I_{j+n+m}.
    \end{align*}
    It follows that $D(L_n \cdot I_m) = \frac{1}{2} \bigl(D(L_n) \cdot I_m + L_n \cdot D(I_m)\bigr)$. This case also imposes no restrictions on $\alpha_j$ and $\beta_j$.
\end{proof}

\begin{corollary}
    Consider the transposed Poisson algebra structure on the Witt algebra induced from $\mathcal{W}(a,-1)$. The space $\frac{1}{2}\textrm{-}\mathfrak{Der}(\mathbf{W})$ is generated by the following operators:
    $$
    A_{t \in \mathbf{Z}} := \mathcal{J}^t.
    $$
\end{corollary}

\subsection{Local derivations}

\begin{definition}\label{def:local-der}
    A linear map $D$ is a \textbf{local derivation} if for every element $x$ in an algebra, there exists a derivation $\Delta_x$ such that:
    $$
    D(x) = \Delta_x(x).
    $$
\end{definition}

\begin{theorem}\label{th:local-derW_a_b}
    The local derivations of the transposed Poisson algebra $\mathcal{W}(a,-1)$ are derivations.

\end{theorem}
\begin{proof}
    Every local derivation is represented as:
    $$D(\blank) = \alpha(\blank)[L_0,\blank] + \beta(\blank)[I_{-a}, \blank] + \lambda(\blank)D_1(\blank),$$
    where $\alpha, \beta, \lambda\colon \mathcal{W}(a,-1) \longrightarrow \mathbf{C}$ are arbitrary maps. Thus, we have two cases:

    \noindent
    \underline{Case 1 $a\in\mathbf{Z}$:}
    \begin{align*}
        D(x)+D(y) &= \alpha(x)[L_0,x] + \alpha(y)[L_0, y]+\beta(x)[I_{-a},x]+\beta(y)[I_{-a},y],\\
        D(x+y) &= \alpha(x+y)[L_0,x] + \alpha(x+y)[L_0, y] \\
        &+ \beta(L_n+y)[I_{-a},x] + \beta(L_n+y)[I_{-a},y].
    \end{align*}
    We get system:
    $$
        \begin{cases}
            \bigl(\alpha(x + y) - \alpha(x)\bigr)[L_0,x] =0,\\
            \bigl(\alpha(x+y) - \alpha(y)\bigr)[L_0,y] =0,\\
            \bigl(\beta(x + y) - \beta(x)\bigr)[I_{-a},x] =0,\\
            \bigl(\beta(x+y) - \beta(y)\bigr)[I_{-a},y] =0.
        \end{cases}
    $$
    From this system we obtain that $\alpha(\blank),\ \beta(\blank)$ are constants for this case.
    \newline
    \underline{Case 2 $a\not\in \mathbf{Z}$:}
    \begin{align*}
        D(x) + D(y) &= \alpha(x)[L_0,x] + \alpha(y)[L_0, y] + \lambda(x)D_1(x) + \lambda(y)D_1(y),\\
        D(x+y) &= \alpha(x+y)[L_0,x] +  \alpha(x+y)[L_0,y] + \lambda(x+y) D_1(x) + \lambda(x+y) D_1(y).
    \end{align*}
    And the system:
    $$
        \begin{cases}
            \bigl(\alpha(x) - \alpha(x+y)\bigr)[L_0, x] = 0,\\
            \bigl(\alpha(y) - \alpha(x+y)\bigr)[L_0, y] = 0,\\
            \bigl(\lambda(x) - \lambda(x+y)\bigr)D_1(x) = 0,\\
            \bigl(\lambda(x) - \lambda(x+y)\bigr)D_1(y) = 0.
        \end{cases}
    $$
    From this system obtain that $\alpha(\blank),\ \lambda(\blank)$ are constants for non-integer $a$. That completes the proof.
\end{proof}
\begin{corollary}
Consider the transposed Poisson algebra structure on the Witt algebra $\mathbf{W}$ induced from $\mathcal{W}(a,-1)$. The local derivations of $\mathbf{W}$ are derivations.
\end{corollary}

\subsection{2-local derivations}

\begin{definition}
    A (not necessarily linear) map $D$ is called a \textbf{2-local derivation} if for any elements $x$ and $y$ in an algebra, there exists a derivation $\Delta_{x,y}$ such that:
    \begin{align*}
        D(x) &= \Delta_{x,y}(x), \\
        D(y) &= \Delta_{x,y}(y).
    \end{align*}
\end{definition}

\begin{theorem}\label{th:2-local-derW_a_b}
    Every 2-local derivation of $\mathcal{W}(a,-1)$ is a derivation. In other words,
    $$2\textrm{-}\mathfrak{loc~Der}\bigl(\mathcal{W}(a,-1)\bigr) = \mathfrak{Der}\bigl(\mathcal{W}(a,-1)\bigr).$$
\end{theorem}

\begin{proof}
    We represent a 2-local derivation as follows:
    $$D(x) = \alpha(x,y)[L_0, x] + \beta(x,y)[I_{-a}, x] + \lambda(x,y) D_1 (x),$$
    where $\alpha, \beta, \lambda$ are scalar functions. For any $y$ and $y'$, we have:
    \begin{align*}
        D(x) &= \alpha(x,y)[L_0, x] + \beta(x,y)[I_{-a}, x] + \lambda(x,y) D_1 (x), \\
        D(x) &= \alpha(x,y')[L_0, x] + \beta(x,y')[I_{-a}, x] + \lambda(x,y') D_1 (x).
    \end{align*}
    Subtracting these equations gives:
    \begin{multline*}
        0 = \bigl(\alpha(x,y)- \alpha(x,y')\bigr)[L_0, x] + \bigl(\beta(x,y)-\beta(x,y')\bigr)[I_{-a}, x]\\
        + \bigl(\lambda(x,y) - \lambda(x,y')\bigr) D_1 (x).
    \end{multline*}
    This identity shows that $\alpha(x,y)$ is independent of $y$. Due to the symmetry between $x$ and $y$, we conclude that $\alpha, \beta$, and $\lambda$ are constants. Thus, $D$ is a derivation.
\end{proof}

\begin{remark}
    The proof technique of this theorem relies purely on the $\mathbb{Z}$-graded structure of $\mathcal{W}(a,-1)$ with one-dimensional homogeneous components. Consequently, this 2-local rigidity result naturally extends to any infinite-dimensional $\mathbb{Z}$-graded algebra whose derivation algebra is spanned by homogeneous elements.
\end{remark}

\begin{corollary}
    Consider the transposed Poisson algebra structure on the Witt algebra $\mathbf{W}$ induced from $\mathcal{W}(a,-1)$. The 2-local derivations of $\mathbf{W}$ are derivations:
    $$2\textrm{-}\mathfrak{loc~Der}(\mathbf{W}) = \mathfrak{Der}(\mathbf{W}).$$
\end{corollary}

\subsection{Quasi-derivations}

\begin{definition}\label{def:quasi-derW_a_b}
    A pair of linear maps $(D, F)$ is called a \textbf{quasi-derivation} if it satisfies the following condition:
    \begin{align*}
        D(x\star y) &= F(x)\star y + x\star F(y),
    \end{align*}
    for all elements $x,y$ and all products $\star$ in the algebra.
\end{definition}

In~\cite{IvanQD}, I.~Kaygorodov, A.~Khudoyberdiyev, and Z.~Shermatova established that the space of quasi-derivations of a Lie structure of $\mathcal{W}(a,-1)$ is the sum of the space of derivations and the space of $\frac{1}{2}$-derivations:
$$\mathfrak{QDer}_{\opn{Lie}}\bigl(\mathcal{W}(a,-1)\bigr) = \mathfrak{Der}_{\opn{Lie}}\bigl(\mathcal{W}(a,-1)\bigr) \oplus \frac{1}{2}\textrm{-}\mathfrak{Der}_{\opn{Lie}} \bigl(\mathcal{W}(a,-1)\bigr).$$

\begin{theorem}\label{th:quasi-derW_a_b}
    Any quasi-derivation $(D, F)$ of the transposed Poisson algebra $\mathcal{W}(a,-1)$ is defined by its action on the generators:
    \begin{align*}
        F(L_n) &= \sum_{j}\alpha_j L_{n+j} + \sum_j \beta_j I_{n+j} - \gamma\,n L_n - \tau n I_{n-a}, \\
        F(I_m) &= \sum_{j} \alpha_j I_{m+j} - (m+a)\,\gamma I_{m} + \lambda  I_m,
    \end{align*}
    and
    \begin{align*}
        D(L_n) &= \sum_j 2\alpha_jL_{n+j} + 2\beta_j I_{n+j} - \gamma\, nL_n - \tau n I_{n-a}, \\
        D(I_m) &= \sum_{j} 2\alpha_j I_{m+j} - (m+a)\,\gamma I_{m} + \lambda  I_m,
    \end{align*}
    where $\tau, \gamma \in \mathbf{C}$ and the sequences $\alpha_j, \beta_j \subset \mathbf{C}$ have finite support.
\end{theorem}
\begin{remark}
    In operator terms, the quasi-derivations $(D, F)$ on $\mathcal{W}(a,-1)$ are generated by:
    \begin{align*}
    A_{t\in\mathbf{Z}} := \begin{pmatrix}
        \mathcal{J}^t & 0 \\
        0 & \mathcal{J}^t
    \end{pmatrix}, \quad 
    \Gamma := \begin{pmatrix}
        \mathrm{diag}\bigl(n\mid n\in\mathbf{Z}\bigr) & 0 \\
        0 & \mathrm{diag}\bigl(m+a\mid m\in\mathbf{Z}\bigr)
    \end{pmatrix}, \\ 
    B_{t\in\mathbf{Z}} := \begin{pmatrix}
        0 & 0 \\
        \mathcal{J}^{t} & 0
    \end{pmatrix}, \quad
        T := \begin{pmatrix}
            0 & 0 \\
            \bigl(i\delta_{i, j+a}\bigr)_{i,j} & 0
        \end{pmatrix}, \quad 
        \Lambda :=  \begin{pmatrix}
            0 & 0 \\
            0 & \mathrm{Id}
        \end{pmatrix}.
    \end{align*}
    These operators satisfy the following relations:
    \begin{align*}
        F &= \sum_t \left(C_t\,A_t + C'_t\,B_t\right) - C''\Gamma - C''' T + \widetilde{C}\,\Lambda, \\
        D &= 2\sum_t \left(C_t\,A_t + C'_t\,B_t\right) - C''\,\Gamma - C'''\,T + \widetilde{C}\,\Lambda,
    \end{align*}
    for arbitrary complex coefficients.
\end{remark}
\begin{proof}
    The expressions for $F$ follow directly from Theorems~\ref{th:derW_a_b} and \ref{th:1/2-derW_a_b} and the results in~\cite{IvanQD}. The action of $D$ is obtained by direct computation using the relations:
    \begin{align*}
        D(L_n) &= F(L_n)\cdot L_0 + L_n \cdot F(L_0), \\
        D(I_m) &= F(I_m)\cdot L_0 + I_m \cdot F(L_0).
    \end{align*}
    Calculating these gives:
    \begin{align*}
        D(L_n) &= \sum_j 2\alpha_jL_{n+j} + 2\beta_j I_{n+j} - \gamma\, nL_n - \tau n I_{n-a}, \\
        D(I_m) &= \sum_{j} 2\alpha_j I_{m+j} - (m+a)\,\gamma I_{m} + \lambda  I_m.
    \end{align*}
    This describes the action of $D$ on the generators.
\end{proof}
\begin{corollary}
    Consider the transposed Poisson algebra structure on the Witt algebra $\mathbf{W}$ induced from $\mathcal{W}(a,-1)$. Any quasi-derivation $(D, F)$ of $\mathbf{W}$ is described by:
    \begin{align*}
        F(L_n) &= \sum_{j}\alpha_j L_{n+j} - \gamma\,n L_n, \\
        D(L_n) &= \sum_j 2\alpha_jL_{n+j} - \gamma\, nL_n,
    \end{align*}
    where $\gamma \in \mathbf{C}$ and the sequence $\alpha_j \subset \mathbf{C}$ has finite support.
\end{corollary}

\subsection{$\delta$-derivations}
\begin{definition}\label{def:delta-derW_a_b}
    A linear map $D$ is called a \textbf{$\delta$-derivation} if it satisfies the following condition:
    $$
        D(x\star y) = \delta\, \bigl(D(x)\star y + x\star D(y)\bigr),
    $$
    where $\delta \in \mathbf{C}$. This condition holds for all elements $x, y$ and all products $\star$ in the algebra. Note that in terms of Definition~\ref{def:quasi-derW_a_b}, this means $F = \delta D$.
\end{definition}
\begin{proposition}
    For the transposed Poisson algebra $\mathcal{W}(a,-1)$ and the induced structure on the Witt algebra, non-trivial $\delta$-derivations exist only for $\delta=1$ or $\delta=\frac{1}{2}$.
\end{proposition}
\begin{proof}
    We solve the equation $F = \delta D$. First, we consider the generator $L_n$:
    \begin{align*}
        F(L_n) &= \sum_{j}\alpha_j L_{n+j} + \sum_j \beta_j I_{n+j} - \gamma\,n L_n - \tau n I_{n-a},\\
        \delta D(L_n) &= \sum_j 2\delta\alpha_jL_{n+j} + 2\delta\beta_j I_{n+j} - \delta\gamma\, nL_n - \delta\tau n I_{n-a}.
    \end{align*}
    Comparing the coordinates, we obtain the following system:
    \begin{align*}
        \alpha_j (1-2\delta) &= 0, \quad j \neq 0, \\
        \alpha_0 (1-2\delta) - \gamma n (1-\delta) &= 0, \\
        \beta_j (1-2\delta) &= 0, \quad j \neq -a, \\
        \beta_{-a} (1-2\delta) - \tau n (1-\delta) &= 0.
    \end{align*}
    If $\delta = \frac{1}{2}$, then the coefficients $\alpha_j$ and $\beta_j$ are arbitrary, while $\tau = 0$ and $\gamma = 0$. If $\delta = 1$, then $\alpha_j = 0$ and $\beta_j = 0$ for all $j$, while $\tau$ and $\gamma$ are arbitrary complex numbers. If $\delta \notin \left\{\frac{1}{2}, 1\right\}$, then $\alpha_j = 0$ for $j \neq 0$ and $\beta_j = 0$ for $j \neq -a$. The remaining equations are:
    \begin{align*}
        \alpha_0 (1-2\delta) - \gamma n (1-\delta) &= 0, \\
        \beta_{-a} (1-2\delta) - \tau n (1-\delta) &= 0.
    \end{align*}
    For $n=0$, we find $\alpha_0 = 0$ and $\beta_{-a} = 0$. Consequently, $\tau = 0$ and $\gamma = 0$. Thus, there are no non-trivial $\delta$-derivations for $\delta \notin \left\{\frac{1}{2}, 1\right\}$.
    
    Next, we check the generator $I_m$:
    \begin{align*}
        F(I_m) &= \sum_{j} \alpha_j I_{m+j} - (m+a)\,\gamma I_{m} + \lambda  I_m,\\
        \delta D(I_m) &= \sum_{j} 2\delta\alpha_j I_{m+j} - \delta(m+a)\,\gamma I_{m} + \delta\lambda  I_m.
    \end{align*}
    We test the cases $\delta = 1$ and $\delta = \frac{1}{2}$. If $\delta = 1$, the equation becomes:
    \begin{align*}
        -(m+a)\gamma I_m + \lambda  I_m &= -(m+a)\gamma I_m + \lambda  I_m.
    \end{align*}
    This identity holds for any $\gamma$ and $\lambda$.\\
    If $\delta = \frac{1}{2}$, the equation becomes:
    \begin{align*}
        \sum_{j} \alpha_j I_{m+j} + \lambda  I_m &= \sum_{j} \alpha_j I_{m+j} + \frac{1}{2}\lambda  I_m.
    \end{align*}
    This implies $\lambda = 0$.
\end{proof}

\section{Automorphisms and Related Operators}
\subsection{Automorphisms}
S.~Gao, C.~Jiang, and Y.~Pei described the automorphisms of the Lie algebra structure on $\mathcal{W}(a,-1)$ in~\cite{SGao}. They established the following isomorphism:
$$\mathfrak{Aut}_{\opn{Lie}}\bigl(\mathcal{W}(a,-1)\bigr) \cong 
\begin{cases}
    \mathbf{C}^\infty \rtimes (\mathbf{C}^* \times \mathbf{C}^*) & a \not\in \frac{1}{2}\mathbf{Z}, \\
    \mathbf{C}^\infty \rtimes (\mathbf{Z}_2 \ltimes (\mathbf{C}^* \times \mathbf{C}^*)) & \text{otherwise}.
\end{cases}
$$
Here, $\mathbf{C}^\infty$ denotes the set of all sequences of complex numbers with finite support. The action on the generators is determined by the following subgroup structures.
\noindent

\subsubsection*{The Normal Abelian Subgroup $\mathfrak{I}$}
The subgroup $\mathfrak{I} := \left\langle \exp(k\,\mathrm{ad}\ I_i) \mid i \in \mathbf{Z}, k \in \mathbf{C} \right\rangle$ of $\mathfrak{Aut}\bigl(\mathcal{W}(a,-1)\bigr)$ acts on the generators as follows:
\begin{align*}
    \prod_{j} \exp(k_{j} \mathrm{ad}\ I_{j}) L_n &= L_n + \sum_{j} k_{j}(j + a - n) I_{j + n}, \\
    \prod_{j} \exp(k_{j} \mathrm{ad}\ I_{j}) I_m &= I_m,
\end{align*}
where $k_j \subset \mathbf{C}$ is a sequence with finite support.

\subsubsection*{Subgroups of Scale and Reflection}
When $a \notin \frac{1}{2}\mathbf{Z}$, the automorphism group is determined by $\mathfrak{a}_1 := \{ \overline{\sigma}_{\alpha,\mu} \mid \alpha, \mu \in \mathbf{C}^* \} \cong \mathbf{C}^* \times \mathbf{C}^*$:
\begin{align*}
    \overline{\sigma}_{\alpha,\mu}(L_n) &= \alpha^n L_n, \\
    \overline{\sigma}_{\alpha,\mu}(I_m) &= \alpha^m \mu I_m.
\end{align*}
For the remaining cases, the structure involves $\mathfrak{a}_2 \cong \mathbf{Z}_2$ and $\mathfrak{b}_2 \cong \mathbf{C}^* \times \mathbf{C}^*$. Their actions are defined as follows:
\begin{enumerate}
    \item \textbf{Action of} $\mathfrak{a}_2 = \{ \overline{\sigma}_{\varepsilon} \mid \varepsilon = \pm 1 \}$:
    $$
    \overline{\sigma}_{\varepsilon}(L_n) =
    \begin{cases}
        \varepsilon L_{\varepsilon n} & a \neq 0 \\
        \varepsilon L_{\varepsilon n} + \lambda_\varepsilon n I_{\varepsilon n} & a = 0
    \end{cases};
    \qquad
    \overline{\sigma}_{\varepsilon}(I_m) 
    =I_{\varepsilon(m + a) - a.}
    $$
    \item \textbf{Action of} $\mathfrak{b}_2 = \{ \overline{\sigma}_{\alpha,\mu} \mid \alpha, \mu \in \mathbf{C}^* \}$:
    $$
    \overline{\sigma}_{\alpha,\mu}(L_n) = 
    \begin{cases} 
        \alpha^n L_n & a \neq 0 \\ 
        \alpha^n L_n + \alpha^n \lambda_1 n I_n & a = 0 
    \end{cases}; \\
    \qquad
    \overline{\sigma}_{\alpha,\mu}(I_m) = \alpha^m \mu I_m.
    $$
\end{enumerate}
\noindent
Using this notation, we express the automorphisms of the Lie structure as follows:
$$\mathfrak{Aut}_{\opn{Lie}}\bigl(\mathcal{W}(a,-1)\bigr) \cong 
\begin{cases}
    \mathfrak{J}\rtimes \mathfrak{a}_1 & a \not\in \frac{1}{2}\mathbf{Z}, \\
    \mathfrak{J}\rtimes (\mathfrak{a}_2 \ltimes \mathfrak{b}_2) & \text{otherwise}.
\end{cases}
$$

\begin{theorem}\label{th:aut-W_a_b}
    The group of automorphisms of the transposed Poisson algebra $\mathcal{W}(a,-1)$ is generated by:
    $$
    \begin{pmatrix}
        \mathrm{diag}\bigl(\alpha^n\mid n\in\mathbf{Z}\bigr) & 0 \\
        0 &  \mathrm{diag}\bigl(\mu\alpha^n\mid n\in\mathbf{Z}\bigr)
    \end{pmatrix},
    $$
    where $\alpha, \mu \in \mathbf{C}^*$. In particular,
    $$
    \mathfrak{Aut}\bigl(\mathcal{W}(a,-1)\bigr) \cong \mathbf{C}^* \times \mathbf{C}^*.
    $$
\end{theorem}
\begin{proof}
    To prove the theorem, we check which Lie automorphisms preserve the associative product $x \cdot y$. Since $\mathcal{W}(a,-1)$ is generated by $L_n$ and $I_m$, it is sufficient to verify the conditions $f(L_n \cdot L_m) = f(L_n) \cdot f(L_m)$ and $f(L_n \cdot I_m) = f(L_n) \cdot f(I_m)$.
    
    \medskip
    \noindent
    \underline{Case 1: $a \notin \frac{1}{2}\mathbf{Z}$.} \\
    We begin by fixing the generators of the Lie algebra automorphisms: $\tau \in \mathfrak{I}$ and $\sigma_{\alpha,\mu} \in \mathfrak{a}_1$. Define the map $f(x) := \tau (\sigma_{\alpha,\mu}(x))$.

    First, consider the action on $L_n$ and $L_m$:
    $$
    f(L_{n+m}) = f(L_n \cdot L_m) = f(L_n) \cdot f(L_m).
    $$
    Expanding both sides:
    \begin{align*}
        f(L_{n+m}) &= \alpha^{n+m} \left(L_{n+m} + \sum_i k_i (i+a-n-m)I_{i+n+m}\right), \\
        f(L_n) \cdot f(L_m) &= \alpha^{n+m} \left(L_n + \sum_i k_i (i+a-n)I_{i+n}\right) \cdot \left(L_m + \sum_i k_i (i+a-m)I_{i+m}\right) \\
        &= \alpha^{n+m} \left(L_{n+m} + \sum_i k_i(2i+2a -n-m)I_{i+n+m}\right).
    \end{align*}
    Comparing the coefficients at $I_{i+n+m}$, we obtain the condition:
    $$k_i (i+a) = 0.$$
    This implies $k_i = c \delta_{i, -a}$ for some $c \in \mathbf{C}$. Thus, $\tau(L_n) = L_n - cnI_{n-a}$ and $\tau(I_m) = I_m$.

    Next, we apply the condition to the product $L_n \cdot I_m$:
    \begin{align*}
        f(I_{n+m}) &= f(L_n \cdot I_m) = f(L_n) \cdot f(I_m), \\
        \alpha^{n+m} \mu I_{n+m} &= \alpha^n (L_n - cn I_{n-a}) \cdot (\alpha^m \mu I_m) \\
        &= \alpha^{n+m} \mu I_{n+m} - \alpha^{n+m} \mu cn I_{n-a+m}.
    \end{align*}
    This forces $c=0$, yielding $f(L_n) = \alpha^n L_n$ and $f(I_m) = \alpha^m \mu I_m$.

    \medskip
    \noindent
    \underline{Case 2: $a \in \frac{1}{2}\mathbf{Z}$ and $a \notin \mathbf{Z}$.} \\
    Let $f(x) := \tau (\sigma_{\varepsilon}(\sigma_{\alpha,\mu}(x)))$ where $\sigma_{\varepsilon} \in \mathfrak{a}_2$ and $\sigma_{\alpha,\mu} \in \mathfrak{b}_2$.
    Repeating the calculation for $L_n$ and $L_m$:
    \begin{align*}
        f(L_{n+m}) &= \varepsilon \alpha^{n+m} \left(L_{\varepsilon(n+m)} + \sum_i k_i (i+a - \varepsilon(n+m))I_{i+\varepsilon(n+m)}\right), \\
        f(L_n) \cdot f(L_m) &= \varepsilon^2 \alpha^{n+m} \left(L_{\varepsilon(n+m)} + \sum_i k_i \bigl[2i + 2a - \varepsilon(n+m)\bigr] I_{i+\varepsilon(n+m)}\right).
    \end{align*}
    Comparing the coefficients of $L_{\varepsilon(n+m)}$ gives $\varepsilon = \varepsilon^2$, hence $\varepsilon = 1$. The coefficients of $I_{i+\varepsilon(n+m)}$ again yield $k_i(i+a) = 0$. The rest of the argument follows Case~1.

    \medskip
    \noindent
    \underline{Case 3: $a \in \mathbf{Z}$.} \\
    It is known from~\cite{SGao} that $\mathcal{W}(x, b) \cong \mathcal{W}(x+n, b)$ for all integers $n$, so we reduce this to $a=0$. Let $f(x) := \tau (\sigma_{\varepsilon}(\sigma_{\alpha,\mu}(x)))$. Following~\cite{SGao}, we have:
    \begin{align*}
        f(L_\xi) &= \tau\left(\alpha^\xi \bigl[\varepsilon L_{\varepsilon \xi} + (\lambda_\varepsilon + \lambda_1)\xi I_{\varepsilon \xi}\bigr]\right) \\
        &= \alpha^\xi \left( \varepsilon \bigl[L_{\varepsilon \xi} + \sum_j k_j(j - \varepsilon \xi) I_{j+\varepsilon\xi}\bigr] + (\lambda_\varepsilon + \lambda_1)\xi I_{\varepsilon\xi} \right).
    \end{align*}
    To check the preservation of the associative product, we expand both sides of equation
    
    \noindent
    $f(L_{n+m}) = f(L_n) \cdot f(L_m)$. Comparing the coefficients of $L_{\varepsilon(n+m)}$, we immediately obtain $\varepsilon = \varepsilon^2$, which forces $\varepsilon = 1$. Substituting $\varepsilon = 1$ back into the expressions, we get:
    \begin{align*}
        f(L_{n+m}) &= \alpha^{n+m} \Biggl(( L_{n+m} + \sum_{j \neq 0} k_j(j - (n+m)) I_{j+n+m} \\
        &+ \bigl[ (\lambda_1 + \lambda_1) - k_0 \bigr] (n+m) I_{n+m} \Biggr), \\
        f(L_n) \cdot f(L_m) &= \alpha^n \left( L_n + \sum_{j} k_j(j-n) I_{j+n} + 2\lambda_1 n I_n \right) \\
        &\cdot \alpha^m \left( L_m + \sum_{j} k_j(j-m) I_{j+m} + 2\lambda_1 m I_m \right).
    \end{align*}
    Using the property $L_p \cdot I_q = I_{p+q}$ and $I_p \cdot I_q = 0$, the product $f(L_n) \cdot f(L_m)$ simplifies to:
    \begin{align*}
        f(L_n) \cdot f(L_m) = \alpha^{n+m} \bigg( L_{n+m} &+ \sum_{j \neq 0} k_j(j-n) I_{j+n+m} + \sum_{j \neq 0} k_j(j-m) I_{j+n+m} \\
        &+ \bigl[ k_0(0-n) + k_0(0-m) + 2\lambda_1 n + 2\lambda_1 m \bigr] I_{n+m} \bigg).
    \end{align*}
    Now, we compare the coefficients for each basis element $I_{j+n+m}$:
    \begin{itemize}
        \item For $j \neq 0$:
        $$ k_j(j - n - m) = k_j(j-n) + k_j(j-m) = k_j(2j - n - m). $$
        This simplifies to $k_j j = 0$, which implies $k_j = 0$ for all $j \neq 0$.
        \item For $j = 0$ (the coefficient of $I_{n+m}$):$$ (2\lambda_1 - k_0)(n+m) = (2\lambda_1 - k_0)(n+m). 
        $$
        But we maintain that $k_0$ also equals zero. Indeed, let us write down the equations for $f(L_n)\cdot f(L_m) = f(L_{n+m})$ for $a\in\mathbf{Z}\setminus \{0\}$:
        \begin{align*}
            f(L_{n+m}) &=\alpha^{n+m}\left(L_{n+m}+\sum_j k_j(j+a-n-m)I_{i+n+m}\right),\\
            f(L_n)\cdot f(L_m) &= \alpha^{n+m}\left(L_{n+m}+\sum_jk_j(2j + 2a - n -m)I_{j+n+m}\right).
        \end{align*}
        So for $j=0$ equation for coefficients of $I_{n+m}$ is $k_j(2a) = 0$ for non-zero $a$. This yields $k_0(n+m) = 0$, so $k_0 = 0$.
    \end{itemize}
    Since $k_j = \lambda_j / \varepsilon_j$ (where $\varepsilon_j$ are non-zero constants from the Lie structure), the vanishing of all $k_j$ implies that all $\lambda_j$ are also zero. Thus, the only automorphisms that preserve the associative product are the scaling maps $\sigma_{\alpha, \mu}$, where $\alpha, \mu \in \mathbf{C}^*$.
\end{proof}

Similar to the case of derivations, the associative product strongly restricts the automorphisms of the algebra. For the pure Lie structure on $\mathcal{W}(a,-1)$, S.~Gao, C.~Jiang, and Y.~Pei~\cite{SGao} found a large automorphism group. Q.~Jiang and S.~Wang~\cite{JQWS} obtained a similarly large group for the deformative Schr\"{o}dinger--Virasoro algebra $\mathbf{W}^g(a,b)$. In contrast, compatibility with the associative product eliminates most of these global maps. This makes the transposed Poisson automorphism group of $\mathcal{W}(a,-1)$ significantly smaller and more rigid than its pure Lie counterparts.

\subsection{Local automorphisms}

\begin{definition}
    A linear map $A$ is called a \textbf{local automorphism} if for each element $x$ there exists an automorphism $\mathcal{A}_x$ such that:
    $$
    A(x) = \mathcal{A}_x(x).
    $$
\end{definition}

\begin{theorem}\label{th:local-aut-W_a_b}
    Local automorphisms of the transposed Poisson algebra $\mathcal{W}(a,-1)$ are automorphisms.
\end{theorem}

\begin{proof}
    The proof follows the same logic as the proof for local derivations. Let $A \in \mathfrak{loc~Aut}(\mathcal{W}(a,-1))$. By definition, for any $x$, the action of $A$ is given by $$A(x) = \bigl(\mathrm{diag}(\alpha^n(x)\mid n\in\mathbf{Z})\oplus\mathrm{diag}(\mu(x)\alpha^n(x)\mid n\in\mathbf{Z})\bigr)x.$$ Applying this to the basis generators, we have:
    \begin{align*}
        A(L_n) &= \alpha^n(L_n) L_n, \\
        A(I_m) &= \mu(I_m)\alpha^m(I_m) I_m.
    \end{align*}
    Thus, from equation $A(x+y) = A(x) + A(y)$, we get the system below:
    \begin{align*}
        A\left(\sum\limits_i x_i L_i + y_i I_i\right) &= \sum\limits_i x_i \alpha^i \left(\sum\limits_j x_j L_j + y_j I_j\right) L_i \\
        &+ y_i \mu\left(\sum\limits_j x_j L_j + y_j I_j\right)\alpha^i\left(\sum\limits_j x_j L_j + y_j I_j\right) I_i,\\
        \sum\limits_i A(x_i L_i) + A(y_i I_i)&= \sum\limits_i x_i \alpha^i (x_i L_i) L_i + y_i \mu(y_i I_i)\alpha^i(y_i I_i) I_i.
    \end{align*}
    First, by comparing the coefficients of $L_i$ we find that $\alpha(\blank)$ is constant. Then, from the coefficients of $I_i$, it follows that $\mu(\blank)$ is constant.
\end{proof}

\begin{corollary}
    Local automorphisms of the transposed Poisson Witt algebra $\mathbf{W}$ are automorphisms.
\end{corollary}

\subsection{2-local automorphisms}

\begin{definition}
    A (not necessarily linear) map $A$ is called a \textbf{2-local automorphism} if for every pair of elements $x, y$ there exists an automorphism $\mathcal{A}_{x,y}$ such that:
    \begin{align*}
    A(x) &= \mathcal{A}_{x,y}(x), \\
    A(y) &= \mathcal{A}_{x,y}(y).
    \end{align*}
\end{definition}

\begin{theorem}\label{th:2-local-aut-W_a_b}
    Every 2-local automorphism of the transposed Poisson algebra $\mathcal{W}(a,-1)$ is an automorphism. In other words, any $A \in 2\textrm{-}\mathfrak{loc~Aut}\left(\mathcal{W}(a,-1)\right)$ has the form:
    $$
    \begin{pmatrix}
        \mathrm{diag}(\alpha^n\mid n\in\mathbf{Z}) & 0 \\
        0 &  \mathrm{diag}(\mu\alpha^n\mid n\in\mathbf{Z})
    \end{pmatrix},
    $$
    for some constants $\alpha, \mu \in \mathbf{C}^*$.
\end{theorem}

\begin{proof}
    Let $A$ be a 2-local automorphism. For any pairs $(x, y)$ and $(x, y')$, there exist automorphisms $\mathcal{A}_{x,y}$ and $\mathcal{A}_{x,y'}$ such that:
    \begin{align*}
    A(x) &= \mathcal{A}_{x,y}(x) = \left( \mathrm{diag}(\alpha^n(x,y)\mid n\in\mathbf{Z}) \oplus \mu(x,y)\mathrm{diag}(\alpha^n(x,y)\mid n\in\mathbf{Z}) \right)x, \\
    A(x) &= \mathcal{A}_{x,y'}(x) = \left( \mathrm{diag}(\alpha^n(x,y')\mid n\in\mathbf{Z}) \oplus \mu(x,y')\mathrm{diag}(\alpha^n(x,y')\mid n\in\mathbf{Z}) \right)x.
    \end{align*}
    Thus, the following difference must be zero:
    \begin{align*}
    \mathrm{diag}\bigl(\alpha^n(x,y) - \alpha^n(x,y')\mid n\in\mathbf{Z}\bigr) \oplus \mathrm{diag}\bigl(\mu(x,y)\alpha^n(x,y) - \mu(x,y')\alpha^n(x,y')\mid n\in\mathbf{Z}\bigr) x.
    \end{align*}
    Setting $x = L_n$ and $y' = L_0$, we get:
    \begin{align*}
    &\alpha^n(L_n, y) - \alpha^n(L_n, L_0) = 0, \\
    &\alpha(L_n, y) = \zeta \alpha(L_n, L_0),
    \end{align*}
    where $\zeta$ is an $n$-th root of unity compatible with the action on $L_0$. Since the left-hand side must be independent of $y$, the function $\alpha$ is constant across the section $L_n \oplus \mathcal{W}(a,-1)$. By the symmetry of the arguments $x$ and $y$, it follows that $\alpha(x,y) \equiv \alpha$ is a global constant in $\mathbf{C}^*$.

    Now consider $x = I_m$ and $y' = L_0$:
    \begin{align*}
        \mu(I_m, y) \alpha^m(I_m, y) &= \mu(I_m, L_0) \alpha^m(I_m, L_0), \\
        \mu(I_m, y) &= \mu(I_m, L_0) \frac{\alpha^m(I_m, L_0)}{\alpha^m(I_m, y)}.
    \end{align*}
    Since $\alpha$ is already proven to be constant, we have $\mu(I_m, y) = \mu(I_m, L_0) = \mu$. Thus, $\alpha(x,y) \equiv \alpha$ and $\mu(x,y) \equiv \mu$ are constants, and $A$ is a global automorphism.
\end{proof}

\begin{corollary}
    Any 2-local automorphism of the transposed Poisson algebra $\mathbf{W}$ is an automorphism.
\end{corollary}

\begin{remark}
    This rigidity result also applies to a wider class of graded algebras. In any $\mathbf{Z}$-graded algebra with one-dimensional components, the grading structure is highly constrained. This rigid framework strongly limits how 2-local maps can behave. As a result, every 2-local automorphism is forced to coincide with a global one. Therefore, 2-local rigidity is a general feature of such infinite-dimensional structures, rather than a special property of $\mathcal{W}(a,-1)$.
\end{remark}

\subsection{Quasi-automorphisms}

\begin{definition}
    A pair of linear maps $(A, F)$ is a \textbf{quasi-automorphism} if it satisfies the following condition:
    $$
    A(x\star y) = F(x)\star F(y),
    $$
    for all elements $x, y$ and all products $\star$ in the algebra.
\end{definition}

Before proceeding to the main theorem, we state a useful lemma.

\begin{lemma}\label{lem:idemp-cdot}
    The only idempotents in the associative structure of the transposed Poisson algebra $\mathcal{W}(a,-1)$ are $0$ and $L_0$.
\end{lemma}
\begin{proof}
    This follows directly from the fact that the associative structure of $\mathcal{W}(a,-1)$ is an algebra of extended Laurent polynomials, where the only solutions to $p^2 = p$ are $0$ and $1$ (corresponding to $L_0$).
\end{proof}

\begin{lemma}\label{lem:units-cdot}
    An element $w\in\mathcal{W}(a,-1)$ is invertible with respect to $\cdot$ if and only if
    $$
    w = \gamma L_k + \sum_j \beta_j I_j, \qquad \gamma\in\mathbf{C}^*,\ k\in\mathbf{Z},
    $$
    with $\beta_j\subset\mathbf{C}$ of finite support.
\end{lemma}
\begin{proof}
    Write $w = P + Q\varepsilon$. If $w$ is invertible, then $P$ is invertible in $\mathbf{C}[t^{\pm1}]$ (reduce modulo $\varepsilon$), so $P = \gamma t^k$ with $\gamma\neq 0$. Conversely, if $P = \gamma t^k$, then $w^{-1} = P^{-1} - P^{-2}Q\varepsilon$.
\end{proof}

\begin{lemma}\label{lem:bracket-dot}
    For all $x, y, w\in\mathcal{W}(a,-1)$,
    $$
    [x,y] = x\cdot[L_0,y] - [L_0,x]\cdot y \qquad\text{and}\qquad [w\cdot x, w\cdot y] = (w\cdot w)\cdot[x,y].
    $$
\end{lemma}
\begin{proof}
    The first identity is checked on the basis: $L_n\cdot(-mL_m) - (-nL_n)\cdot L_m = (n-m)L_{n+m}$, $\ L_n\cdot\bigl(-(m+a)I_m\bigr) - (-nL_n)\cdot I_m = -(m+a-n)I_{n+m}$, and both sides vanish for $x = I_n$, $y = I_m$ (the case $x = I_n$, $y = L_m$ follows by antisymmetry). Put $D := -\mathrm{ad}_{L_0}$; by Theorem~\ref{th:derW_a_b} it is a derivation of $\cdot$, and $[x,y] = D(x)\cdot y - x\cdot D(y)$. Therefore, using that $\cdot$ is commutative and associative,
    \begin{align*}
        [w\cdot x, w\cdot y] &= \bigl(D(w)\cdot x + w\cdot D(x)\bigr)\cdot w\cdot y - w\cdot x\cdot\bigl(D(w)\cdot y + w\cdot D(y)\bigr)\\
        &= w\cdot w\cdot\bigl(D(x)\cdot y - x\cdot D(y)\bigr) = (w\cdot w)\cdot[x,y].
    \end{align*}
\end{proof}

\begin{theorem}\label{th:quasi-auth-W_a_b}
    A pair $(A, F)$ is a quasi-automorphism of $\mathcal{W}(a,-1)$ if and only if
    $$
    F = R_w\circ f, \qquad A = R_{w\cdot w}\circ f,
    $$
    where $w$ is an invertible element of $\bigl(\mathcal{W}(a,-1),\cdot\bigr)$ and $f$ is an automorphism of $\mathcal{W}(a,-1)$. The pair $(w, f)$ is uniquely determined by $(A,F)$; namely, $w = F(L_0)$. Explicitly, by Lemma~\ref{lem:units-cdot} and Theorem~\ref{th:aut-W_a_b}, with $w = \gamma L_k + \sum_j\beta_j I_j$ and $f = \tau_c\circ\sigma_{\alpha,\mu}$:
    \begin{align*}
        F(L_n) &= \alpha^n\Bigl(\gamma L_{n+k} + \sum_j \beta_j I_{n+j} - c\gamma\, n\, I_{n+k-a}\Bigr), &
        F(I_m) &= \gamma\mu\alpha^m I_{m+k},\\
        A(L_n) &= \alpha^n\Bigl(\gamma^2 L_{n+2k} + 2\gamma\sum_j \beta_j I_{n+j+k} - c\gamma^2 n\, I_{n+2k-a}\Bigr), &
        A(I_m) &= \gamma^2\mu\alpha^m I_{m+2k},
    \end{align*}
    where $\gamma,\alpha,\mu\in\mathbf{C}^*$, $k\in\mathbf{Z}$, $c\in\mathbf{C}$ (the terms with $c$ vanish for $a\notin\mathbf{Z}$), and $\beta_j\subset\mathbf{C}$ has finite support.

    Moreover, $A = F$ if and only if $w = L_0$, that is, if and only if $A = F = f$ is an automorphism. The group of quasi-automorphisms is isomorphic to $\mathcal{U}\rtimes\mathfrak{Aut}\bigl(\mathcal{W}(a,-1)\bigr)$, where $\mathcal{U}$ is the group of invertible elements of $\bigl(\mathcal{W}(a,-1),\cdot\bigr)$, with the multiplication $(w,f)(w',f') = \bigl(w\cdot f(w'), f f'\bigr)$.
\end{theorem}

\begin{proof}
    Let $w$ be invertible and $f$ an automorphism, and put $F = R_w\circ f$, $A = R_{w\cdot w}\circ f$. These maps are bijective, since $R_w^{-1} = R_{w^{-1}}$. For the associative product,
    $$
    A(x\cdot y) = w\cdot w\cdot f(x)\cdot f(y) = \bigl(w\cdot f(x)\bigr)\cdot\bigl(w\cdot f(y)\bigr) = F(x)\cdot F(y),
    $$
    and for the bracket, by Lemma~\ref{lem:bracket-dot},
    $$
    A([x,y]) = (w\cdot w)\cdot[f(x),f(y)] = [w\cdot f(x), w\cdot f(y)] = [F(x),F(y)].
    $$
    Hence $(A,F)$ is a quasi-automorphism.

    Conversely, let $(A,F)$ be a quasi-automorphism and put $w := F(L_0)$. Since $L_0$ is the unit of $\cdot$,
    $$
    A(x) = A(x\cdot L_0) = F(x)\cdot F(L_0) = w\cdot F(x),
    $$
    that is, $A = R_w\circ F$. Note that this uses only the defining identity, with $A$ on the left; we never use $F(x\cdot y) = F(x)\cdot F(y)$, which need not hold. Since $A$ is surjective, $L_0 = A(v) = w\cdot F(v)$ for some $v$, so $w$ is invertible. Put $f := R_{w^{-1}}\circ F$. Then $f$ is bijective, $F = R_w\circ f$ and $A = R_w\circ F = R_{w\cdot w}\circ f$. Now
    $$
    (w\cdot w)\cdot f(x\cdot y) = A(x\cdot y) = F(x)\cdot F(y) = (w\cdot w)\cdot\bigl(f(x)\cdot f(y)\bigr),
    $$
    and, by Lemma~\ref{lem:bracket-dot},
    $$
    (w\cdot w)\cdot f([x,y]) = A([x,y]) = [w\cdot f(x), w\cdot f(y)] = (w\cdot w)\cdot[f(x),f(y)].
    $$
    Since $w\cdot w$ is invertible, $f$ preserves both products, so $f$ is an automorphism, and it is described by Theorem~\ref{th:aut-W_a_b}. Every automorphism fixes $L_0$ (as $\tau_c(L_0) = \sigma_{\alpha,\mu}(L_0) = L_0$), so $F(L_0) = w\cdot f(L_0) = w$, which proves uniqueness. The explicit formulas follow from $w\cdot L_n = \gamma L_{n+k} + \sum_j\beta_j I_{n+j}$, $w\cdot I_m = \gamma I_{m+k}$ and $w\cdot w = \gamma^2 L_{2k} + 2\gamma\sum_j\beta_j I_{j+k}$.

    If $A = F$, then $w\cdot w = A(L_0) = F(L_0) = w$, so $w$ is an invertible idempotent, and $w = L_0$ by Lemma~\ref{lem:idemp-cdot}. Conversely, $w = L_0$ gives $A = F = f$. Finally, since $f$ preserves $\cdot$, we have $f\circ R_{w'} = R_{f(w')}\circ f$, hence $(R_w f)(R_{w'} f') = R_{w\cdot f(w')}\, f f'$ and similarly for $A$; this gives the group law.
\end{proof}

From Theorems \ref{th:quasi-auth-W_a_b} and \ref{th:local-aut-W_a_b} we obtain the following result.

\begin{corollary}
    Local quasi-automorphisms of the transposed Poisson algebra $\mathcal{W}(a,-1)$ are defined by:$$\mathfrak{loc~QAut}\bigl(\mathcal{W}(a,-1)\bigr) = { \begin{pmatrix}\mathrm{diag}\bigl(\alpha^n_{L}(n)\mid n\in\mathbf{Z}\bigr) & 0 \\
    0 & \mathrm{diag}\bigl(\mu(m)\alpha^m_I(m)\mid m\in\mathbf{Z}\bigr)\end{pmatrix} },$$where $\alpha_L, \alpha_I, \mu\colon \mathbf{Z} \longrightarrow \mathbf{C}^*$ are arbitrary functions.
\end{corollary}

Similarly, from Theorems \ref{th:quasi-auth-W_a_b} and \ref{th:2-local-aut-W_a_b} we get the following result.
\begin{corollary}
    Every 2-local quasi-automorphism of the transposed Poisson algebra $\mathcal{W}(a,-1)$ is an automorphism.
\end{corollary}

\section{Homogeneous Rota--Baxter operators on $\mathcal{W}(a,-1)$}

We now turn to a class of operators originally introduced by Glen Baxter in his analytic proof of the Spitzer identity~\cite{GBaxterOriginal}. By introducing these operators, Baxter reduced a non-trivial probabilistic problem to the verification of an algebraic identity. The most intuitive example of such an operator is the integration operator:
$$I \colon C(\mathbf{R}) \longrightarrow C(\mathbf{R}), \quad I(f)\big|_x := \int\limits_0^x f(t) \, dt.$$

Let $F(x) := I(f)\big|_x$ and $G(x) := I(g)\big|_x$ be the primitives of $f$ and $g$, respectively. Multiplying $F$ by $G$ and applying the integration-by-parts formula, we obtain:
$$ F(x)G(x) = I(f)\big|_x I(g)\big|_x \stackrel{(*)}{=} I(f \, I(g))\big|_x + I(I(f)\, g)\big|_x = \int\limits_0^x f(t)G(t) \, dt + \int\limits_0^x F(t)g(t) \, dt.$$
\noindent
Equation $(*)$ corresponds to the Rota--Baxter condition of weight zero. 

We formalize this class of operators as follows:
\begin{definition}\label{def:RotBaxterOp}
    A linear map $R$ is called a \textbf{Rota--Baxter operator} of weight $\lambda \in \mathbf{C}$ if it satisfies the \textbf{Rota--Baxter binding equation}:\
    $$
    R(x) \star R(y) = R\bigl(R(x) \star y + x \star R(y)\bigr) + \lambda R (x \star y).
    $$
    Note that if $R$ is a Rota--Baxter operator of weight $\lambda \neq 0$, then $\frac{1}{\lambda}R$ is a Rota--Baxter operator of weight $1$.
\end{definition}

\begin{definition}\label{def:HomogeneousOp}
    Let $E_\bullet = \bigoplus_{n \in \mathbf{Z}} E_n$ be a graded algebra. A linear map $L \colon E_\bullet \longrightarrow E_\bullet$ is called a \textbf{homogeneous operator of degree $k \in\mathbf{Z}$} if $L(E_n) \leq E_{n+k}$ for all $n \in \mathbf{Z}$.
    It's easy to see that this definition can be extrapolated on all ordered set with additive structure.
\end{definition}

In light of Definition~\ref{def:RotBaxterOp}, it is sufficient to classify operators of weight $0$ and $1$. For the Witt algebra $\mathbf{W}$, a complete classification of homogeneous Rota--Baxter operators is provided in~\cite{RotaBaxterWitt}. Every homogeneous operator of degree $k$ on the Witt algebra takes the form $L_n \mapsto f(n) L_{n+k}$ for some function $f \colon \mathbf{Z} \longrightarrow \mathbf{C}$. We list below the results from~\cite{RotaBaxterWitt} that are necessary for our generalization to $\mathcal{W}(a,-1)$.

\begin{theorem}\label{fact:Rot-Baxter-0-w}
    There are exactly three types of homogeneous Rota--Baxter operators of weight $0$ on $\mathbf{W}$:
    \begin{alignat}{4}
        &R^{\alpha}_{k}(L_n) &&\mathrel{:=} \alpha \delta_{n+2k, 0} L_{n+k}, \qquad & & k \in \mathbf{Z},\ \alpha \in \mathbf{C}, \label{eq:zero-weight-RB-1} \\[4pt]
        &R'^{\beta}_{2k}(L_n) &&\mathrel{:=} \beta \bigl(\delta_{n+2k, 0} + 2\delta_{n+3k, 0}\bigr) L_{n+2k}, \qquad & & k \in \mathbf{Z}^*,\ \beta \in \mathbf{C}^*, \label{eq:zero-weight-RB-2} \\[4pt]
        &R^{l,\gamma}_{k}(L_n) &&\mathrel{:=} \frac{k\gamma}{n+2k}\, \delta_{ n+k \in l\mathbf{Z}}\, L_{n+k}, \qquad & & k, l \in \mathbf{Z}^*,\ l \nmid k,\ \gamma \in \mathbf{C}^*. \label{eq:zero-weight-RB-3}
    \end{alignat}
\end{theorem}

\begin{theorem}\label{fact:Rot-Baxter-1-w}
    All homogeneous Rota--Baxter operators of weight $1$ on $\mathbf{W}$ have degree $0$ and belong to one of the following classes (where $\alpha \in \mathbf{C}$):
    \begin{alignat*}{4}
        &\bullet R^{\leq 1}_{0}(L_n) &&:= \begin{cases} -L_n & n \geq 2 \\ 0 & n \leq 1 \end{cases} 
        &\hspace{2em}&\bullet R^{\geq -1}_{0}(L_n) &&:= \begin{cases} -L_n & n \leq -2 \\ 0 & n \geq -1 \end{cases} \\[3ex]
        &\bullet R^{> 1}_{0}(L_n) &&:= \begin{cases} -L_n & n \leq 1 \\ 0 & n \geq 2 \end{cases}
        &\hspace{2em}&\bullet R^{< -1}_{0}(L_n) &&:= \begin{cases} -L_n & n \geq -1 \\ 0 & n \leq -2 \end{cases} \\[3ex]
        &\bullet R^{+,\alpha}_{0}(L_n) &&:= \begin{cases} -L_n & n < 0 \\ \alpha L_0 & n = 0 \\ 0 & n > 0 \end{cases}
        &\hspace{2em}&\bullet R^{-,\alpha}_{0}(L_n) &&:= \begin{cases} -L_n & n > 0 \\ \alpha L_0 & n = 0 \\ 0 & n < 0 \end{cases} \\[3ex]
        &\bullet R^{\emptyset}_{0}(L_n) &&:= -L_n
        &\hspace{2em}&\bullet R^{0}_{0}(L_n) &&:= 0
    \end{alignat*}
\end{theorem}

\subsection{Homogeneous Rota--Baxter operators on $\mathcal{W}(a,-1)$ with $\mathbf{Z}$-grading}
Firstly, we fix the grading on the algebra $\mathcal{W}(a,-1)$. In this section, we consider the following $\mathbf{Z}$-grading:$$    \mathcal{W}(a,-1) = \bigoplus_{n \in \mathbf{Z}} \bigl(\mathbf{C} L_n + \mathbf{C} I_n\bigr).$$

\begin{theorem}\label{th:RB-0-w-Z-grad-Lie}
    Let $R$ be a homogeneous block-diagonal Rota--Baxter operator of weight $0$ on the Lie structure of $\mathcal{W}(a,-1)$ with the matrix representation $R\big|_{\mathbf{W}}\oplus R\big|_{I(a,-1)}$ (invariant on $\mathbf{W}$ and $I(a,-1)$). So the possible operators are given by the following relations:
    \begin{subequations}
    \setlength{\jot}{8pt}
    \begin{align}
        &\left\{\begin{aligned}
            R^{\alpha}_{k}(L_n) &:= \alpha \delta_{n+2k, 0} L_{n+k}, \\
            R^{\alpha}_{k}(I_m) &:= \tau \delta_{m, -a-2k} \delta_{k, 0} I_{m+k},
        \end{aligned}\right. && k \in \mathbf{Z}, \; \alpha, \tau \in \mathbf{C}. \label{eq:homRB_Z_1}\\[5pt]
        &\left\{\begin{aligned}
            R'^{\beta}_{2k}(L_n) &:= \beta (\delta_{n+2k, 0} + 2\delta_{n+3k, 0}) L_{n+2k}, \\
            R'^{\beta}_{2k}(I_m) &:= 0,
        \end{aligned}\right. && k \in \mathbf{Z}^*, \; \beta \in \mathbf{C}^*.\label{eq:homRB_Z_2} \\[5pt]
        &\left\{\begin{aligned}
            R^{l,\gamma,m_0}_{k}(L_n) &:= \frac{k\gamma}{n+2k} \delta_{n+k \in l\mathbf{Z}} L_{n+k}, \\[2pt]
            R^{l,\gamma,m_0}_{k}(I_m) &:= \frac{k\gamma}{m+a+2k} \delta_{m \in m_0+l\mathbf{Z}} \, \delta_{m \neq -a-2k} I_{m+k},
        \end{aligned}\right. && (*). \label{eq:homRB_Z_3}
    \end{align}
    \end{subequations}
    \noindent where for case~\eqref{eq:homRB_Z_3}, $k, l \in \mathbf{Z}^*$, $l \nmid k$, $\gamma \in \mathbf{C}^*$, and $m_0 \in \mathbf{Z}$ satisfies $m_0 \not= -2k-a \pmod{l}$. We adopt the convention that $\frac{\delta_{x \neq 0}}{x} = 0$ at $x=0$.
\end{theorem}
\begin{proof}
    Suppose a homogeneous Rota--Baxter operator $R$ of degree $k$ takes the form:
    $$ R(L_n) := \zeta(n) L_{n+k}, \quad R(I_m) := \xi(m) I_{m+k}. $$
    Due to the invariance of $R$ on $\mathbf{W}$ and $I(a,-1)$, it suffices to verify the Rota--Baxter condition on the bracket $[L_n, I_m]$:
    \begin{align}
        [R(L_n), R(I_m)] &= -(m+a-n)\zeta(n)\xi(m)I_{n+m+2k}, \notag \\
        R\bigl([R(L_n), I_m] + [L_n, R(I_m)]\bigr) 
            &= -\bigl((m+a-n-k)\zeta(n) \notag\\
            &\qquad + (m+k+a-n)\xi(m)\bigr)\xi(m+n+k)I_{m+n+2k}. \label{eq:RB-functional}
    \end{align}
    Setting $n = -k$, we obtain:$$ \bigl(k\zeta(-k) - (m+a+2k)\xi(m)\bigr)\xi(m) = 0. $$
    This implies either $\xi(m) = 0$ or $\xi(m) = \frac{k\zeta(-k)}{m+a+2k}$ for $m \neq -a-2k$. Let $\varepsilon_m \subset \{0, 1\}^{\mathbf{Z}}$ be a selector sequence. The general form is:
    $$ \xi(m) = \frac{k\zeta(-k)}{m+a+2k} \varepsilon_m \delta_{m, \neq -a-2k} + \tau \delta_{m, -a-2k}, $$
    where $\tau \in \mathbf{C}$, and $\tau \neq 0$ only if $k\zeta(-k) = 0$. Using the known classification for $\zeta(n)$ from Theorem~\ref{fact:Rot-Baxter-0-w}, we analyze the three types:
    \begin{itemize}
        \item \underline{Case~\eqref{eq:homRB_Z_1}}. Here $\zeta(n) = \alpha \delta_{n, -2k}$. If $k \neq 0$, then $\zeta(-k) = 0$, forcing $\xi(m) = 0$ for almost all $m$. If $k=0$, we obtain the identity $R(I_m) = \tau \delta_{m, -a} I_m$.
        \item \underline{Case~\eqref{eq:homRB_Z_2}}. In this case, the degree of the operator $R$ is $2k$, and the function on the $\mathbf{W}$ subspace is given by $\zeta(n) = \beta(\delta_{n, -2k} + 2\delta_{n, -3k})$. To determine the behavior of $\xi(m)$, we substitute these into the mixed Rota--Baxter functional equation~\eqref{eq:RB-functional}.
        \begin{itemize}
            \item First, let $n = -2k$. Since $\zeta(-2k) = \beta$, equation~\eqref{eq:RB-functional} simplifies to the identity we used to establish the general form: 
            $$ \xi(m) = \frac{2k\beta}{m+a+4k} \varepsilon_m. $$
            \item Next, we consider $n = -3k$. Here, $\zeta(-3k) = 2\beta$. A crucial shift occurs in the index of the right-hand side of~\eqref{eq:RB-functional}: the term $\xi(m+n+2k)$ becomes $\xi(m-k)$. This couples the values of $\xi$ at points $m$ and $m-k$, yielding:
            \begin{align*}
                (m+a+3k)(2\beta)\xi(m) &= \bigl((m+a+k)(2\beta) + (m+5k+a)\xi(m)\bigr)\xi(m-k).
            \end{align*}
            In this formula $m+a+4k\not= 0$. If $m+a+4k= 0$ then obtain $k\epsilon_{a-5k}=0$.
            Substituting the fractional form of $\xi(m)$ and $\xi(m-k)$ and dividing by the common factor $4k\beta^2$, we arrive at the following recurrence relation for the selector sequence:
            \begin{equation}\label{eq:recurrence}
                (m+a+3k)^2 \varepsilon_m = \bigl( (m+a+k)(m+a+4k) + k(m+a+5k)\varepsilon_m \bigr) \varepsilon_{m-k}.
            \end{equation}
            Analyzing the support of $\varepsilon_m$ reveals three scenarios:
            \begin{enumerate}
                \item If $\varepsilon_m = 1$ and $\varepsilon_{m-k} = 0$, the equation implies $(m+a+3k)^2 = 0$, which only happens if $m = -a-3k$.\label{item:hRB:n=-3k_1}
                \item If $\varepsilon_m = 0$ and $\varepsilon_{m-k} = 1$, the right side must vanish, implying $m = -a-k$ or $m = -a-4k$.\label{item:hRB:n=-3k_2}
                \item If $\varepsilon_m = 1$ and $\varepsilon_{m-k} = 1$, the equation~\eqref{eq:recurrence} is satisfied identically, theoretically allowing for infinite sequences of non-zero values.\label{item:hRB:n=-3k_3}
            \end{enumerate}
            However, we must also satisfy the Rota--Baxter condition for all $n \notin \{-2k, -3k\}$. For these indices, $\zeta(n) = 0$, and the functional equation reduces to:
            $$ 0 = (m+2k+a-n)\epsilon_m \epsilon_{m+n+2k}. $$
            If there exists any $m_0$ such that $\epsilon_{m_0} \neq 0$, then for almost all $n$, we must have $\epsilon_{m_0+n+2k} = 0$. This ``orthogonality'' condition prevents the existence of the chains described in the third case or the multiple isolated points in the first or the second cases. Therefore, the only globally consistent selector sequence is the trivial one, $\varepsilon_m \equiv 0$, which implies $\xi \equiv 0$.
            \end{itemize}
        \item \underline{Case~\eqref{eq:homRB_Z_3}}. Here $\zeta(n) = \frac{k\gamma}{n+2k} \delta_{n+k \in l\mathbf{Z}}$.
        \begin{itemize}
            \item If $n+k \notin l\mathbf{Z}$, then $\zeta(n)=0$, and the functional equation~\eqref{eq:RB-functional} becomes:
            $$ 0 = (m+k+a-n)\varepsilon_m \varepsilon_{m+n+k} \quad (\text{for } m, m+n+k \neq -a-2k). $$
            If we assume $\varepsilon_{m_0} = 1$ for some $m_0 \neq -a-2k$, then for any $n \in \mathbf{Z} \setminus \{-a-3k-m_0, -k+l\mathbf{Z}\}$, we must have $\varepsilon_{m_0+k+n} = 0$. This implies that the support of $\varepsilon_m$ is very sparse outside of the cosets related to $l\mathbf{Z}$.
            \item Let $n+k = lt$ for some $t \in \mathbf{Z}$. Substituting $\zeta(n)$ and $\xi(m)$ into~\eqref{eq:RB-functional} and setting $m = -a-2k$, we get $0 = (2k+lt)k^2\gamma \varepsilon_{-a-2k+lt} \delta_{lt \neq 0}$. Since $k \nmid l$, we have $2k+lt \neq 0$, implying $\varepsilon_{-a-2k+lt} = 0$ for all $t \in \mathbf{Z}^*$. Thus, $\varepsilon$ can be non-zero at index $-a-2k$ only as an isolated point.
            
            For $m \neq -a-2k-lt$, direct computation shows:
            \begin{align*}
                &(m+a+lt+2k)(m+a-lt+k)\varepsilon_m \\
                &= \bigl((m+a-lt)(m+a+2k) + (m+2k+a-lt)(lt+k)\varepsilon_m\bigr)\varepsilon_{m+lt}.
            \end{align*}
            This equation is satisfied if $\varepsilon_m = \varepsilon_{m+lt} = 1$. This confirms that if any $\varepsilon_{m_0}=1$, then the whole coset $m_0 + l\mathbf{Z}$ should be in the support.
        \end{itemize}
        To complete the argument for this case, we need to match the periodicity from the second subcase with the sparsity from the first subcase.
        
        If there exists some $m_0$ such that $\varepsilon_{m_0} = 1$, the second subcase implies that the whole coset $m_0 + l\mathbf{Z}$ must be contained in the support of $\varepsilon_m$. However, the first subcase imposes a strict ``orthogonality'' condition: for any $m$ in the support, almost all other points $m+k+n$ (when $n+k \notin l\mathbf{Z}$) cannot be in the support.

        For both rules to work together, the support cannot cover more than one coset or have extra points. Any point $x$ outside the coset $m_0 + l\mathbf{Z}$ (except perhaps the single point at $-a-2k$) would eventually break the equation for large shifts $n$.
        
        Also, our analysis of the singular point showed that $\varepsilon_{-a-2k+lt} = 0$ for all $t \neq 0$. This means that $-a-2k$ cannot be part of a full coset $l\mathbf{Z}$. Thus, the only valid and non-zero solution is for the support to be exactly one coset:
        $$ \varepsilon_m = \delta_{m \in m_0 + l\mathbf{Z}}, $$
        where we must ensure $m_0 \not\equiv -a-2k \pmod{l}$ to avoid problems at the singular point. This leads us to the final form of the operator $R^{l,\gamma,m_0}_{k}$.
    \end{itemize}
\end{proof}

\begin{corollary}
    In the transposed Poisson structure on $\mathcal{W}(a,-1)$, every $\mathbf{W}$ and $I(a,-1)$-invariant Rota--Baxter operator of weight $0$ vanishes on $\mathbf{W}$ and has following form:
    $$
    \begin{cases}
    R(L_n) &= 0 \\
    R(I_m) &= \tau \delta_{m, -a-2k} \delta_{k, 0} I_{m+k}
    \end{cases} \quad k \in \mathbf{Z},\, \tau \in \mathbf{C},\, a \notin \mathbf{Z}.
    $$
\end{corollary}
\begin{proof}
    We must verify the Rota--Baxter identity for the associative product $\cdot$:
    $$ R(x) \cdot R(y) = R(R(x) \cdot y + x \cdot R(y)). $$

    If we assume $R(L_n) = \zeta(n) L_{n+k}$, the identity for $L_n \cdot L_m$ becomes $$\zeta(n)\zeta(m) = \zeta(n+m+k)(\zeta(n) + \zeta(m)).$$
    For $R^\alpha_k$, setting $n=m=-2k$ yields $\alpha^2 = \alpha^2(\delta_{-k,0} + \delta_{-k,0})$, which implies $\alpha^2 = 2\alpha^2 \delta_{k,0}$. For any $k \neq 0$, this forces $\alpha = 0$. If $k=0$, the only idempotent solution is $\alpha=0$ or $\alpha=1$, but the latter fails for $n \neq m$.
    Similar contradictions ($1=2$ or singular values) arise for $R'^\beta$ and $R^{l, \gamma}$.

    Finally, we check the mixed product $R(L_n) \cdot R(I_m)$ with $R(L_n)=0$. The condition reduces to $0 = R(L_n \cdot R(I_m))$. 
    Substituting $R(I_m) = \tau \delta_{m, -a} I_m$ (for $k=0$), we get:
    $$ 0 = \tau \delta_{m, -a} R(I_{n+m}) = \tau^2 \delta_{m, -a} \delta_{n+m, -a} I_{n+m}. $$
    This must hold for all $n$. If $n=0$, we have $\tau^2 \delta_{m, -a} = 0$. If $a \notin \mathbf{Z}$, the index $m$ (which is an integer) can never equal $-a$, so the delta-symbol is always zero and the condition is satisfied. If $a \in \mathbf{Z}$, the operator fails at $m=-a$.
\end{proof}

\begin{theorem}
    Let $R$ be a homogeneous block-diagonal Rota--Baxter operator of weight $1$ on the transposed Poisson structure of $\mathcal{W}(a,-1)$ with $\mathbf{Z}$-grading. Suppose $R$ is invariant on $\mathbf{W}$ and $I(a,-1)$, with the matrix representation $R\big|_{\mathbf{W}}\oplus R\big|_{I(a,-1)}$. Then $R$ is one of the following:
    \begin{enumerate}
        \item $R_0^0 \oplus (-\mathrm{Id})$ \text{or} $R_0^0 \oplus 0$;
        \item $R_0^\emptyset \oplus (-\mathrm{Id})$ \text{or} $R_0^\emptyset \oplus 0$;
        \item $\begin{pmatrix} 
                R_0^{+,\alpha} & 0 \\ 
                0 & \mathrm{diag}\bigl(-\mathbf{1}_{(-\infty, t]}(n)\mid n\in\mathbf{Z}\bigr) 
              \end{pmatrix}$ \text{for} $t\in\mathbf{Z}\sqcup\{\pm\infty\}$ \text{and} $\alpha \in \{-1, 0\}$;
        \item $\begin{pmatrix} 
                R_0^{-,\alpha} & 0 \\ 
                0 & \mathrm{diag}\bigl(-\mathbf{1}_{[t, +\infty)}(n)\mid n\in\mathbf{Z}\bigr) 
              \end{pmatrix}$ \text{for} $t\in\mathbf{Z}\sqcup\{\pm\infty\}$ \text{and} $\alpha \in \{-1, 0\}$.
    \end{enumerate}
    \noindent Here, the operators acting on the Witt algebra $\mathbf{W}$ are taken from Theorem~\ref{fact:Rot-Baxter-1-w}.
\end{theorem}
\begin{proof}
    Consider a homogeneous operator $R$ of degree $0$. Let $R(L_n) := \zeta(n)L_n$ and $R(I_m) := \xi(m)I_m$. To verify the Rota--Baxter condition on the transposed Poisson structure, we must satisfy the identity for both the Lie bracket and the associative product.
    
    For the mixed relations between $L_n$ and $I_m$, the Lie condition and the associative condition yield:
    \begin{align}
        (m+a-n)\bigl(\xi(m+n)(\zeta(n)+\xi(m)+1) - \zeta(n)\xi(m)\bigr) &= 0, \label{eq:RB-Lie-mixed} \tag{M-Lie} \\
        \xi(m+n)(\zeta(n)+\xi(m)+1) - \zeta(n)\xi(m) &= 0. \label{eq:RB-Assoc-mixed} \tag{M-Assoc}
    \end{align}
    Since~\eqref{eq:RB-Assoc-mixed} implies~\eqref{eq:RB-Lie-mixed}, we only need to solve the associative functional equation.

    First, we check the compatibility of the Witt algebra operators from Theorem~\ref{fact:Rot-Baxter-1-w} with the associative product. 
    \begin{lemma}
        Among the weight $1$ Rota--Baxter operators on the Lie Witt algebra $\mathbf{W}$ (Theorem~\ref{fact:Rot-Baxter-1-w}), only $R_0^0, R^\emptyset_0, R^{+,\alpha}_0,$ and $R^{-,\alpha}_0$ for $\alpha \in \{-1, 0\}$ are compatible with the associative product.
    \end{lemma} 
    \begin{proof}
        The associative condition $L_n \cdot L_m = L_{n+m}$ requires $\zeta(n)\zeta(m) = \zeta(n+m)(\zeta(n) + \zeta(m) + 1)$. For $n=m$, this gives the diagonal constraint:
        \begin{align}\label{eq:W-Diag}
            \zeta^2(n) = (2\zeta(n) + 1)\zeta(2n). \tag{W-Diag}
        \end{align} 
        Testing the known Lie classification against~\eqref{eq:W-Diag} reveals the following contradictions:
        \begin{itemize}
            \item Operator $R^{\leq 1}_{0}$: Defined by $\zeta(n) = -1$ for $n \geq 2$ and $0$ otherwise. At $n=1$, $\zeta(1)=0$ and $\zeta(2)=-1$. Equation~\eqref{eq:W-Diag} gives $0 = (0+1)(-1)$, a contradiction.
            \item Operator $R^{\geq -1}_{0}$: Defined by $\zeta(n) = -1$ for $n \leq -2$. At $n=-1$, we get $0 = -1$.
            \item Operator $R^{> 1}_{0}$: Defined by $\zeta(n) = -1$ for $n \leq 1$. At $n=1$, $\zeta(1)=-1$ and $\zeta(2)=0$. Equation~\eqref{eq:W-Diag} gives $1 = (-2+1)(0)$, a contradiction.
            \item Operator $R^{< -1}_{0}$: Defined by $\zeta(n) = -1$ for $n \geq -1$. At $n=-1$, we get $1 = 0$.
        \end{itemize}
        For the step operators $R^{\pm, \alpha}_0$, setting $n=0$ in~\eqref{eq:W-Diag} yields $\alpha^2 = (2\alpha + 1)\alpha$, which simplifies to $\alpha(\alpha+1)=0$. Thus, only $\alpha \in \{0, -1\}$ are valid. The cases $R_0^0$ ($\zeta \equiv 0$) and $R^\emptyset_0$ ($\zeta \equiv -1$) satisfy the identity for all $n$.
    \end{proof}

    Using this lemma, we analyze the support of $\xi(m)$ by setting $n=0$ in~\eqref{eq:RB-Assoc-mixed}:
    $$ \xi(m)(\zeta(0) + \xi(m) + 1) - \zeta(0)\xi(m) = 0 \implies \xi(m)(\xi(m)+1) = 0. $$
    Thus, $\xi(m) \in \{0, -1\}$ for all $m$, meaning $(-1)R\big|_{I(a,-1)}$ is a projection onto some subset of indices $D \subset \mathbf{Z}$. To find the shape of $D$, we analyze the cases for $\zeta(n)$:

    \begin{itemize}
        \item Cases $R_0^0$ and $R_0^\emptyset$: If $\zeta \equiv 0$ or $\zeta \equiv -1$, then~\eqref{eq:RB-Assoc-mixed} simplifies to $\xi(m+n)(\xi(m)+1) = 0$ or $\xi(m+n)\xi(m) = -\xi(m)$ respectively. In both cases, $\xi$ must be a constant $0$ or $-1$. This gives the first two items of the theorem.
        \item Case $R_0^{+, \alpha}$: Here $\zeta(n) = -1$ for $n < 0$. If we pick $n < 0$ in~\eqref{eq:RB-Assoc-mixed}, the equation becomes $\xi(m+n)\xi(m) = -\xi(m)$. 
        If $\xi(m) = -1$ for some index $m$, then $\xi(m+n)$ must also be $-1$ for all $n < 0$. This forces the support of $\xi$ to be the form $(-\infty, t]$. 
        Where $t$ is a parametr, so the general form $\mathrm{diag}\bigl(-\mathbf{1}_{(-\infty, t]}\bigr)$.
        \item Case $R_0^{-, \alpha}$: Here $\zeta(n) = -1$ for $n > 0$. By symmetric logic, if $\xi(m) = -1$, then $\xi(m+n)$ must be $-1$ for all $n > 0$. This forces the support to be the form $[t, +\infty)$.
    \end{itemize}
    In both ray cases, $t = \pm \infty$ corresponds to the constant operators $0$ and $-\mathrm{Id}$ respectively.
\end{proof}

\subsection{Homogeneous Rota--Baxter operators on $\mathcal{W}(a,-1)$ with $\mathbf{Z}_2$-grading}
Let us fix a $\mathbf{Z}_2$-grading on our algebra:
$$\mathcal{W}(a,-1) =  \mathbf{W} \oplus I(a,-1).$$
We consider homogeneous Rota--Baxter operators of degree $1$:
$$ R(\mathbf{W}) \leq I(a,-1), \quad R(I(a,-1)) \leq \mathbf{W}. $$
In terms of the basis generators, we write:
$$ R(L_n) = \sum_i \zeta^n_i I_i, \quad R(I_m) = \sum_j \xi^m_j L_j. $$
\begin{theorem}\label{th:RB-1-w-Z-2-grad-Lie}
    There are no non-trivial homogeneous Rota--Baxter operators of weight $1$ on the $\mathbf{Z}_2$-graded transposed Poisson algebra $\mathcal{W}(a,-1)$.
\end{theorem}
\begin{proof}
    We analyze the Rota--Baxter identity $[R(x), R(y)] = R\bigl([R(x), y] + [x, R(y)]\bigr) + R([x, y])$ for different pairs of generators.
    \begin{itemize}
        \item \underline{Case 1 ($[L_n, L_m]$)}. Evaluating the RB identity on the pair $(L_n, L_m)$:
        $$ [R(L_n), R(L_m)] = R\bigl([R(L_n), L_m] + [L_n, R(L_m)]\bigr) + R([L_n, L_m]).$$
        Since $R(L_k) \in I(a,-1)$ and the bracket $[I, I]$ is zero, the left-hand side vanishes:
        \begin{align*}
            0 &= R\left( \left[\sum_i \zeta^n_i I_i, L_m\right] + \left[L_n, \sum_j \zeta^m_j I_j\right] + (n-m)L_{n+m} \right) \\
            &= R\left( -\sum_i \zeta^n_i(i+a-m)I_{i+m} + \sum_j \zeta^m_j(j+a-n)I_{j+n} + (n-m)L_{n+m} \right).
        \end{align*}
        Applying the action of $R$ and splitting the result into the $\mathbf{W}$ and $I$ components, we get
        $$ 0 = \underbrace{\left( \sum_{i,k} \zeta^m_i(i+a-n)\xi^{i+n}_k L_k - \sum_{i,k} \zeta^n_i(i+a-m)\xi^{i+m}_k L_k \right)}_{\in \mathbf{W}} + \underbrace{(n-m)\sum_k \zeta^{n+m}_k I_k}_{\in I(a,-1)}. $$
        For the $I$-component to vanish for all $n, m$, we must have $(n-m)\zeta^{n+m}_k = 0$. For any index $s$, we can choose $n, m$ such that $n+m=s$ and $n \neq m$, which forces $\zeta^s_k = 0$ for all $s, k$. Thus, $R(\mathbf{W}) = 0$.
        \item \underline{Case 2 ($[L_n, I_m]$)}. Now we apply the RB identity to the pair $(L_n, I_m)$:
        $$ [R(L_n), R(I_m)] = R\bigl([R(L_n), I_m] + [L_n, R(I_m)]\bigr) + R([L_n, I_m]). $$
        Since $R(L_n) = 0$, the left-hand side is $0$. Also $[R(L_n), I_m] = 0$. Therefore:
        $$ 0 = R\bigl( 0 + [L_n, R(I_m)] + [L_n, I_m] \bigr). $$
        Now expand the terms inside $R$. After computation we get:
        \begin{multline*}
            0 = R\left( \left[L_n, \sum_i \xi^m_i L_i\right] - (m+a-n)I_{n+m} \right)\\
            = R\left( \sum_i \xi^m_i (n-i)L_{n+i} \right) - (m+a-n)R(I_{n+m}).
        \end{multline*}
        Since $R(L_j) \equiv 0$, the first sum vanishes, leaving:
        $$ (m+a-n) \sum_i \xi^{m+n}_i L_i = 0. $$
        This must hold for all $n, m$. For a fixed sum $s = m+n$, we have $(s - 2n + a)\xi^s_i = 0$. Since we can vary $n$ freely, the coefficient $(s-2n+a)$ cannot be zero for all $n$. Thus, $\xi^s_i = 0$ for all $s, i$, which implies $R(I) = 0$.
    \end{itemize}
\end{proof}
\begin{corollary}
    There is no non-trivial homogeneous Rota--Baxter operator of weight $1$ on the $\mathbf{Z}_2$-graded Lie algebra $\mathcal{W}(a,-1)$.
\end{corollary}

\section{$\mathbf{W}$-compatible Novikov--Poisson structures}
As shown by X.~Kong, H.~Chen, and C.~Bai in~\cite{ClassLSoverW}, every Novikov structure compatible with the Witt algebra $\mathbf{W}$ belongs to the family $V_{\alpha,0}$ for $\alpha \in \mathbf{C}$ (where $V_{\alpha,0} \cong V_{-\alpha,0}$). The Novikov product $\circ$ in these structures is defined as:
$$L_n \circ L_m := (\alpha + m) L_{n+m}.$$

Our goal is to see if this product $\circ$ and the associative product $L_n \cdot L_m := L_{n+m}$ together form a Novikov--Poisson structure. To do this, we must verify two compatibility conditions:
\begin{align}
    (x \cdot y) \circ z &= x \cdot (y \circ z), \label{eq:NP1} \tag{NP1} \\
    (x \circ y) \cdot z - x \circ (y \cdot z) &= (y \circ x) \cdot z - y \circ (x \cdot z). \label{eq:NP2} \tag{NP2}
\end{align}
It turns out that these rules hold for any choice of $\alpha$.
\begin{theorem}
    For any complex number $\alpha$, the Novikov product $\circ$ and the associative product $\cdot$ define a $\mathbf{W}$-compatible Novikov--Poisson structure.
\end{theorem}
\begin{proof}
    We verify this by direct computation. For the first condition~\eqref{eq:NP1}:
    $$
    \begin{aligned}
        (L_i \cdot L_j) \circ L_k &= L_{i+j} \circ L_k &= (\alpha + k) L_{i+j+k}, \\
        L_i \cdot (L_j \circ L_k) &= L_i \cdot (\alpha + k) L_{j+k} &= (\alpha + k) L_{i+j+k}.
    \end{aligned}
    $$
    Both sides are equal.

    For the second condition~\eqref{eq:NP2} the results also match, thus the condition holds:
    $$
    \begin{aligned}
        (L_i \circ L_j) \cdot L_k - L_i \circ (L_j \cdot L_k) &= (\alpha + j) L_{i+j+k} - (\alpha + j + k) L_{i+j+k} &= -k L_{i+j+k}, \\
        (L_j \circ L_i) \cdot L_k - L_j \circ (L_i \cdot L_k) &= (\alpha + i) L_{i+j+k} - (\alpha + i + k) L_{i+j+k} &= -k L_{i+j+k}.
    \end{aligned}
    $$
\end{proof}
\begin{corollary}
    The commutator of the Novikov product defines a Lie bracket for this Novikov–Poisson structure. This bracket is consistent with the definition of $\mathcal{W}(a,-1)$. Specifically, the triple: $ \left(\mathbf{W}, \cdot, [\ ,\ ]_\alpha\right) $ is a transposed Poisson algebra where the Lie product is given by:
    $$[L_n, L_m]_\alpha := L_n \circ L_m - L_m \circ L_n = (m - n) L_{n+m}.$$
\end{corollary}
    Note that although the Novikov product depends on $\alpha$, the parameter $\alpha$ cancels out in the commutator. This results in the standard Lie bracket of the Witt algebra.

{\small\bibliography{cimart}}

@article{AKS26,
author = {Abdelwahab, Hani and Kaygorodov, Ivan and Sartayev, B. K.},
journal = {J. Non-Associative Struct.},
pages = {9},
title = {{$\delta$-Poisson and transposed $\delta$-Poisson algebras}},
volume = {1},
year = {2026}
}

@article{AAES_24,
author = {Abdurasulov, Kobiljon and Adashev, Jobir and Eshmeteva, Sabohat},
doi = {10.46298/cm.12844},
issn = {2336-1298},
journal = {Commun. Math.},
language = {english},
number = {3},
pages = {441--483},
title = {{Transposed Poisson structures on solvable Lie algebras with filiform nilradical}},
url = {https://cm.episciences.org/12844},
volume = {32},
year = {2024}
}

@article{ABR25,
author = {Andruszkiewicz, Ryszard R. and Brzezi{\'{n}}ski, Tomasz and Radziszewski, Krzysztof},
doi = {10.1007/s00025-025-02377-7},
issn = {1422-6383},
journal = {Results Math.},
number = {2},
pages = {61},
title = {{Lie Affgebras Vis-{\`{a}}-Vis Lie Algebras}},
url = {https://link.springer.com/10.1007/s00025-025-02377-7},
volume = {80},
year = {2025}
}

@article{AEK2023,
author = {Ayupov, Shavkat and Elduque, Alberto and Kudaybergenov, Karimbergen},
doi = {10.1016/j.jpaa.2022.107277},
issn = {00224049},
journal = {J. Pure Appl. Algebr.},
number = {5},
pages = {107277},
title = {{Local derivations and automorphisms of Cayley algebras}},
url = {https://linkinghub.elsevier.com/retrieve/pii/S0022404922002754},
volume = {227},
year = {2023}
}

@article{BBGW23,
author = {Bai, Chengming and Bai, Ruipu and Guo, Li and Wu, Yong},
doi = {10.1016/j.jalgebra.2023.06.006},
issn = {00218693},
journal = {J. Algebr.},
language = {english},
pages = {535--566},
title = {{Transposed Poisson algebras, Novikov--Poisson algebras and $3$-Lie algebras}},
url = {https://linkinghub.elsevier.com/retrieve/pii/S0021869323002910},
volume = {632},
year = {2023}
}

@article{GBaxterOriginal,
author = {Baxter, Glen},
doi = {10.2140/pjm.1960.10.731},
issn = {0030-8730},
journal = {Pacific J. Math.},
language = {english},
number = {3},
pages = {731--742},
title = {{An analytic problem whose solution follows from a simple algebraic identity}},
url = {http://msp.org/pjm/1960/10-3/p01.xhtml},
volume = {10},
year = {1960}
}

@article{BS26,
author = {Belwal, Pragya and Singh, Mahender},
doi = {10.46298/jonas.17223},
journal = {J. Non-Associative Struct.},
pages = {1},
title = {{Skew braces and Rota--Baxter operators on semi-direct products}},
url = {https://jonas.episciences.org/17223},
volume = {1},
year = {2026}
}

@article{IntroChen,
author = {Chen, Siyuan and Bai, Chengming},
doi = {10.1088/1751-8121/ad9128},
issn = {1751-8113},
journal = {J. Phys. A Math. Theor.},
language = {english},
number = {49},
pages = {495203},
title = {{Quantizations of transposed Poisson algebras by Novikov deformations}},
url = {https://iopscience.iop.org/article/10.1088/1751-8121/ad9128},
volume = {57},
year = {2024}
}

@article{DQT26,
author = {Daukeyeva, Nigora and Eraliyeva, Maqpal and Toshtemirova, Feruza},
doi = {10.46298/cm.15974},
issn = {2336-1298},
journal = {Commun. Math.},
language = {english},
number = {1},
pages = {1},
title = {{Transposed $\delta$-Poisson algebra structures on null-filiform associative algebras}},
url = {https://cm.episciences.org/15974},
volume = {34},
year = {2026}
}

@article{FKL_21,
author = {Ferreira, Bruno Leonardo Macedo and Kaygorodov, Ivan and Lopatkin, Viktor},
doi = {10.1007/s13398-021-01088-2},
issn = {1578-7303},
journal = {Rev. la Real Acad. Ciencias Exactas, F{\'{i}}sicas y Nat. Ser. A. Matem{\'{a}}ticas},
language = {english},
number = {3},
pages = {142},
title = {{$\frac{1}{2}$-derivations of Lie algebras and transposed Poisson algebras}},
url = {https://link.springer.com/10.1007/s13398-021-01088-2},
volume = {115},
year = {2021}
}

@article{SGao,
author = {Gao, Shoulan and Jiang, Cuipo and Pei, Yufeng},
doi = {10.1080/00927871003591835},
issn = {0092-7872},
journal = {Commun. Algebr.},
language = {english},
number = {2},
pages = {397--423},
title = {{Low-Dimensional Cohomology Groups of the Lie Algebras $W( a , b )$}},
url = {http://www.tandfonline.com/doi/abs/10.1080/00927871003591835},
volume = {39},
year = {2011}
}

@article{RotaBaxterWitt,
author = {Gao, Xu and Liu, Ming and Bai, Chengming and Jing, Naihuan},
doi = {10.1016/j.geomphys.2016.06.007},
issn = {03930440},
journal = {J. Geom. Phys.},
language = {english},
pages = {1--20},
title = {{Rota--Baxter operators on Witt and Virasoro algebras}},
url = {https://linkinghub.elsevier.com/retrieve/pii/S0393044016301486},
volume = {108},
year = {2016}
}

@article{IntroGelf,
author = {Gel'fand, I. M. and Dorfman, I. Ya.},
doi = {10.1007/BF01078363},
issn = {0016-2663},
journal = {Funct. Anal. Its Appl.},
language = {english},
number = {4},
pages = {248--262},
title = {{Hamiltonian operators and algebraic structures related to them}},
url = {http://link.springer.com/10.1007/BF01078363},
volume = {13},
year = {1980}
}

@article{k23,
author = {Kaygorodov, Ivan},
doi = {10.46298/cm.11419},
issn = {2336-1298},
journal = {Commun. Math.},
language = {english},
number = {3},
pages = {1--62},
title = {{Non-associative algebraic structures: classification and structure}},
url = {http://cm.episciences.org/11419},
volume = {32},
year = {2024}
}

@article{IntroIvan,
author = {Kaygorodov, Ivan and Khrypchenko, Mykola},
doi = {10.1016/j.jalgebra.2024.02.033},
issn = {00218693},
journal = {J. Algebr.},
language = {english},
pages = {458--491},
title = {{Transposed Poisson structures on Lie incidence algebras}},
url = {https://linkinghub.elsevier.com/retrieve/pii/S0021869324001224},
volume = {647},
year = {2024}
}

@article{IvanQD,
author = {Kaygorodov, Ivan and Khudoyberdiyev, Abror and Shermatova, Zarina},
doi = {10.1007/s40687-025-00596-6},
issn = {2522-0144},
journal = {Res. Math. Sci.},
language = {english},
number = {1},
pages = {20},
title = {{Quasi-derivations of Witt and related algebras}},
url = {https://link.springer.com/10.1007/s40687-025-00596-6},
volume = {13},
year = {2026}
}

@article{KKS2025,
author = {Kaygorodov, Ivan and Khudoyberdiyev, Abror and Shermatova, Zarina},
doi = {10.1007/s40590-024-00702-8},
issn = {1405-213X},
journal = {Bolet{\'{i}}n la Soc. Matem{\'{a}}tica Mex.},
language = {english},
number = {1},
pages = {22},
title = {{Transposed Poisson structures on not-finitely graded Witt-type algebras}},
url = {https://link.springer.com/10.1007/s40590-024-00702-8},
volume = {31},
year = {2025}
}

@article{KKS25,
author = {Kaygorodov, Ivan and Khudoyberdiyev, Abror and Shermatova, Zarina},
doi = {10.1016/j.geomphys.2024.105356},
issn = {03930440},
journal = {J. Geom. Phys.},
pages = {105356},
title = {{Transposed Poisson structures on Virasoro-type algebras}},
url = {https://linkinghub.elsevier.com/retrieve/pii/S0393044024002572},
volume = {207},
year = {2025}
}

@article{ClassLSoverW,
author = {Kong, Xiaoli and Chen, Hongjia and Bai, Chengming},
doi = {10.1142/S0129167X11006751},
issn = {0129-167X},
journal = {Int. J. Math.},
language = {english},
number = {2},
pages = {201--222},
title = {{Classification of Graded Left-symmetric Algebra Structures on Witt and Virasoro Algebras}},
url = {https://www.worldscientific.com/doi/abs/10.1142/S0129167X11006751},
volume = {22},
year = {2011}
}

@book{Landsman_ClassicalQuantum_1998,
address = {New York, NY},
author = {Landsman, N. P.},
doi = {10.1007/978-1-4612-1680-3},
isbn = {978-1-4612-7242-7},
language = {english},
publisher = {Springer New York},
series = {Springer Monographs in Mathematics},
title = {{Mathematical Topics Between Classical and Quantum Mechanics}},
url = {http://link.springer.com/10.1007/978-1-4612-1680-3},
year = {1998}
}

@article{Leger2000,
author = {Leger, George F. and Luks, Eugene M.},
doi = {10.1006/jabr.1999.8250},
issn = {00218693},
journal = {J. Algebr.},
language = {english},
number = {1},
pages = {165--203},
title = {{Generalized Derivations of Lie Algebras}},
url = {https://linkinghub.elsevier.com/retrieve/pii/S0021869399982509},
volume = {228},
year = {2000}
}

@article{Mikhailov_Vanhaecke_CommutativePoisson_2024,
author = {Mikhailov, Alexander V. and Vanhaecke, Pol},
doi = {10.1007/s11005-024-01855-3},
issn = {1573-0530},
journal = {Lett. Math. Phys.},
language = {english},
number = {5},
pages = {108},
title = {{Commutative Poisson algebras from deformations of noncommutative algebras}},
url = {https://link.springer.com/10.1007/s11005-024-01855-3},
volume = {114},
year = {2024}
}

@article{Sartayev2023,
author = {Sartayev, B. K.},
doi = {10.46298/cm.11346},
issn = {2336-1298},
journal = {Commun. Math.},
language = {english},
number = {2},
pages = {55--62},
title = {{Some generalizations of the variety of transposed Poisson algebras}},
url = {http://cm.episciences.org/11346},
volume = {32},
year = {2024}
}

@article{Sinha_Yadav_PoissonGeometric_2024,
author = {Sinha, Pritish and Yadav, Ankit},
doi = {10.1063/5.0206781},
issn = {0022-2488},
journal = {J. Math. Phys.},
language = {english},
number = {6},
pages = {062101},
title = {{Poisson geometric formulation of quantum mechanics}},
url = {https://pubs.aip.org/jmp/article/65/6/062101/3295606/Poisson-geometric-formulation-of-quantum-mechanics},
volume = {65},
year = {2024}
}

@article{XTang,
author = {Tang, Xiaomin},
doi = {10.11650/tjm/180305},
issn = {1027-5487},
journal = {Taiwan. J. Math.},
language = {english},
number = {6},
pages = {1347--1366},
title = {{Biderivations and Commutative Post-Lie Algebra Structures on the Lie Algebra $\mathcal{W}(a,b)$}},
url = {https://projecteuclid.org/journals/taiwanese-journal-of-mathematics/volume-22/issue-6/Biderivations-and-Commutative-Post-Lie-Algebra-Structures-on-the-Lie/10.11650/tjm/180305.full},
volume = {22},
year = {2018}
}

@article{IntroXu,
author = {Xu, Xiaoping},
doi = {10.1006/jabr.1996.6911},
issn = {00218693},
journal = {J. Algebr.},
language = {english},
number = {2},
pages = {253--279},
title = {{Novikov--Poisson Algebras}},
url = {https://linkinghub.elsevier.com/retrieve/pii/S0021869396969112},
volume = {190},
year = {1997}
}

@article{YXK_23,
author = {Yang, Yang and Tang, Xiaomin and Khudoyberdiyev, Abror},
doi = {10.1142/S1005386726000131},
issn = {1005-3867},
journal = {Algebr. Colloq.},
language = {english},
number = {1},
pages = {127--140},
title = {{Transposed Poisson Structures on Schr{\"{o}}dinger Algebra in $(n+1)$-Dimensional Space-Time}},
url = {https://www.ssrn.com/abstract=4377379 https://www.worldscientific.com/doi/10.1142/S1005386726000131},
volume = {33},
year = {2026}
}

@article{Yoshida_NambuMechanics_2024,
author = {Yoshida, Zensho},
doi = {10.1093/ptep/ptac096},
issn = {2050-3911},
journal = {Prog. Theor. Exp. Phys.},
language = {english},
number = {3},
pages = {03A103},
title = {{Nambu mechanics viewed as a Clebsch parameterized Poisson algebra: Toward canonicalization and quantization}},
url = {https://academic.oup.com/ptep/article/doi/10.1093/ptep/ptac096/6626034},
volume = {2024},
year = {2024}
}

@article{ZZ24,
author = {Zohrabi, Arezoo and Zusmanovich, Pasha},
doi = {10.46298/cm.13595},
issn = {2336-1298},
journal = {Commun. Math.},
language = {english},
number = {1},
pages = {2},
title = {{A $\delta$-first Whitehead Lemma for Jordan algebras}},
url = {https://cm.episciences.org/13595},
volume = {33},
year = {2025}
}

@article{JQWS,
author = {Jiang, Qifen and Wang, Song},
doi = {10.1142/S1005386715000450},
issn = {1005-3867},
journal = {Algebr. Colloq.},
language = {english},
number = {3},
pages = {517--540},
title = {{Derivations and Automorphism Group of Original Deformative Schr{\"{o}}dinger--Virasoro Algebra}},
volume = {22},
year = {2015}
}

\appendix
\section{Differences from the published version}
In Theorems~\ref{th:local-derW_a_b},~\ref{th:derW_a_b},~\ref{th:local-aut-W_a_b}, and~\ref{th:quasi-auth-W_a_b}, the formulations of the theorems have been refined: it is now explicitly stated that the local structures coincide with the global ones, and the derivations and quasi-automorphisms are specified accordingly. In Theorems~\ref{th:aut-W_a_b} and~\ref{th:quasi-auth-W_a_b}, the computations substantiating the validity of the proof for corner cases are presented explicitly.

\end{document}